\documentclass[12pt]{amsart}
\usepackage{amsmath,times,epsfig,amssymb,amsbsy,amscd,amsfonts,amstext,color,bm,mathrsfs}
\usepackage[margin=1in]{geometry}
\usepackage{enumerate}
\usepackage{placeins}
\usepackage{array}
\usepackage{tabulary}
\usepackage{multirow}
\usepackage{graphicx}
\usepackage{hyperref}
\usepackage{hhline}
\usepackage{dashbox}
\usepackage{soul}
\usepackage{multirow}
\usepackage[table, svgnames, dvipsnames]{xcolor}
\usepackage{makecell, cellspace, caption}
\usepackage{subcaption}
\usepackage{todonotes}
\usepackage[color,matrix,arrow]{xypic}
\usepackage{tikz}
        \usetikzlibrary{arrows.meta}
        \usetikzlibrary{arrows}
        \usetikzlibrary{quotes}
        \usetikzlibrary{graphs}
        \usetikzlibrary{positioning}
\usepackage{tikz-cd}
\hypersetup{
  colorlinks   = true, %Colours links instead of ugly boxes
  urlcolor     = blue, %Colour for external hyperlinks
  linkcolor    = blue, %Colour of internal links
  citecolor   = red %Colour of citations
}

\tikzstyle{v} = [circle, draw, inner sep=1.5pt, minimum size=3pt, fill=black]
\theoremstyle{plain}
\newtheorem{theorem}{Theorem}[section]
\newtheorem{corollary}[theorem]{Corollary}
\newtheorem{lemma}[theorem]{Lemma}
\newtheorem{claim}[theorem]{Claim}

\newtheorem{proposition}[theorem]{Proposition}
\newtheorem{question}[theorem]{Question}

\makeatletter
    
    \@addtoreset{equation}{section}
  \makeatother

\theoremstyle{definition}
\newtheorem{definition}[theorem]{Definition}
\newtheorem{notation}[theorem]{Notation}
\newtheorem{def-prop}[theorem]{Definition \& Proposition}
\newtheorem{assumption}[theorem]{Assumption \& Notation}
\newtheorem{example}[theorem]{Example}
\newtheorem{conjecture}[theorem]{Conjecture}

\theoremstyle{remark}
\newtheorem{remark}[theorem]{Remark}

\newcommand{\A}{\mathcal{A}}
\newcommand{\B}{\mathcal{B}}

\newcommand{\R}{\mathbb{R}}
\newcommand{\Qc}{\mathcal{Q}}
\newcommand{\scR}{\mathcal{R}}

\newcommand{\Pc}{\mathcal{P}}

\newcommand{\Z}{\mathbb{Z}}

\newcommand{\V}{{\mathcal{V}}}
\newcommand{\E}{{\mathcal{E}}}

\newcommand{\I}{{\mathcal{I}}}

\newcommand{\LRV}{\mathsf{LRV}}
\newcommand{\RV}{\mathsf{RV}}
\newcommand{\MG}{\mathsf{MG}}
\newcommand{\MCG}{\mathsf{MCG}}

\newcommand{\mcK}{\mathcal{K}}
\newcommand{\tbf}{\textbf}

\newcommand{\rk}{\operatorname{rk}}

\renewcommand{\part}{\operatorname{par}}
\newcommand{\nbd}{\operatorname{Nbd}}

\newcommand{\hooklongrightarrow}{\lhook\joinrel\longrightarrow}

\newcolumntype{K}[1]{>{\centering\arraybackslash}p{#1}}

\allowdisplaybreaks

\makeatletter
\@namedef{subjclassname@2020}{%
  \textup{2020} Mathematics Subject Classification}
\makeatother

\begin{document}

\title[Vines and MAT-labeled graphs]{Vines and MAT-labeled graphs}

\begin{abstract}
The present paper explores a connection between two concepts arising from different fields of mathematics. The first concept, called vine, is a graphical model for dependent random variables. This concept first appeared in a work of Joe (1994), and the formal definition was given later by Cooke (1997). Vines have nowadays become an active research area whose applications can be found in probability theory and uncertainty analysis. The second concept, called MAT-freeness, is a combinatorial property in the theory of  freeness of logarithmic derivation modules of hyperplane arrangements. This concept  was first studied by Abe-Barakat-Cuntz-Hoge-Terao (2016), and soon afterwards  investigated further  by Cuntz-M{\"u}cksch (2020). 

In the particular case of graphic arrangements, the last two authors (2023) recently proved that the MAT-freeness is completely characterized by the existence of certain edge-labeled graphs, called MAT-labeled graphs. In this paper, we first introduce a poset characterization of a vine, the so-called vine. Then we show that, interestingly, there exists an explicit equivalence between the categories of locally regular vines and MAT-labeled graphs. In particular, we obtain an equivalence between the categories of regular vines and MAT-labeled complete graphs. 

Several applications will be mentioned to illustrate the interaction between the two concepts. Notably, we give an affirmative answer to a question of Cuntz-M{\"u}cksch that MAT-freeness can be characterized by a generalization of the root poset in the case of graphic arrangements.
 \end{abstract}

\author{Hung Manh Tran}
\address{Hung Manh Tran, Department of Mathematics, National University of Singapore - 10, Lower Kent Ridge Road -- Singapore 119076}
\email{e0511873@u.nus.edu}

\author{Tan Nhat Tran}
\address{Tan Nhat Tran, Institut f\"ur Algebra, Zahlentheorie und Diskrete Mathematik, Fakult\"at f\"ur Mathematik und Physik, Leibniz Universit\"at Hannover, Welfengarten 1, D-30167 Hannover, Germany.}
\email{tan.tran@math.uni-hannover.de}

\author{Shuhei Tsujie}
\address{Shuhei Tsujie, Department of Mathematics, Hokkaido University of Education, Asahikawa, Hokkaido 070-8621, Japan}
\email{tsujie.shuhei@a.hokkyodai.ac.jp}


\date{\today}
\maketitle

\tableofcontents

%********************************************************************************************************
\section{Introduction}
\label{sec:intro}

%******************************************************************************** 
\subsection{Motivation}
\label{subsec:mov}
The starting point of our paper is a question of Cuntz-M{\"u}cksch  \cite{CM20} (Question \ref{ques:CM}) in the theory of \emph{free hyperplane arrangements}. 

Let $V$ be a finite dimensional vector space. 
A \tbf{hyperplane} in $V$ is a $1$-codimensional linear subspace of $V$. 
Let $ \{x_{1}, \dots, x_{\ell}\} $ be a basis for the dual space $V^{\ast} $. 
Any hyperplane  in $V$ can be described by a linear equation of the form $a_1x_1+\cdots+a_\ell x_\ell=0$ where at least one of the $a_{i}$'s is non-zero. 

A \tbf{hyperplane arrangement} $\A$ is a finite set of  hyperplanes in $V$.
The  \tbf{intersection lattice} of   $\A$ is the set of all intersections of hyperplanes in $\A$, which is often referred to as the   \tbf{combinatorics} of $\A$. 
An arrangement is said to be  \tbf{free} if its  \emph{module of logarithmic derivations} is a free module.  
For basic definitions and properties of free arrangements, we refer the interested reader to \cite{T80, OT92}. 
Freeness is an algebraic property of hyperplane arrangements which has been a major topic of research since the 1970s. 
A central question in the theory is to study the freeness of an arrangement by combinatorial structures, especially by the intersection lattice of the arrangement. 

Among others, \emph{MAT-freeness} is an important concept which was first used by Abe-Barakat-Cuntz-Hoge-Terao \cite{ABCHT16} to settle the conjecture of Sommers-Tymoczko \cite{ST06} on the freeness of \emph{ideal subarrangements of Weyl arrangements}. 
This concept is formally defined later by Cuntz-M{\"u}cksch  \cite{CM20} and we will use their definition throughout.
For a hyperplane $H \in \A$, define the \textbf{restriction} ${\A}^{H}$ of ${\A}$ to $H$ by 
$${\A}^{H}:= \{ K \cap H \mid K \in \A \setminus\{H\}\}.$$
\begin{definition}[{MAT-partition and MAT-free arrangement \cite{CM20}}]
\label{def:MAT-partition-free}
Let $\A$ be a nonempty arrangement. 
A partition (disjoint union of nonempty subsets) $\pi= (\pi_{1}, \dots, \pi_{n}) $ of $\A$ is called an \textbf{MAT-partition} if  the following three conditions hold for every $1 \le k \le n $. 
\begin{enumerate}[(MP1)]
\item\label{definition MAT partition 1} The hyperplanes in $\pi_k$ are linearly independent.
\item\label{definition MAT partition 2} There does not exist $ H^{\prime} \in \mathcal{A}_{k-1} $ such that $ \bigcap_{H \in \pi_{k}}H \subseteq H^{\prime} $, where $ \mathcal{A}_{k-1} := \pi_{1} \sqcup \dots \sqcup \pi_{k-1} $ (disjoint union) and $ \A_{0} := \emptyset$ is the empty arrangement. 
\item\label{definition MAT partition 3} For each $ H \in \pi_{k} $, $ |\mathcal{A}_{k-1}| - |(\mathcal{A}_{k-1} \cup \{H\})^{H}| = k-1 $. 
\end{enumerate}
An arrangement  is called \textbf{MAT-free} if it is empty or  admits an MAT-partition. 
\end{definition}

As the name suggests, any MAT-free arrangement is a  free arrangement. 
This follows from the remarkable Multiple Addition Theorem by Abe-Barakat-Cuntz-Hoge-Terao \cite[Theorem 3.1]{ABCHT16} (justifying the abbreviation MAT). 
MAT-freeness is a helpful combinatorial tool (as it depends only on the intersection lattice) to examine the freeness of arrangements. 
One of its most famous applications we mentioned earlier is a proof that the ideal subarrangements of Weyl arrangements are free. 
The MAT-freeness has received increasing attention in recent years, see  \cite{AT16, AT19, MR21, CK22} for some other applications.

Let  $V = \Bbb R^\ell$ with the standard inner product $(\cdot,\cdot)$.
 Let $\Phi$ be an irreducible (crystallographic) root system in $V$, with a fixed positive system $\Phi^+ \subseteq \Phi$ and the associated set of simple roots $\Delta := \{\alpha_1,\ldots,\alpha_\ell \}$. 
 For  $\alpha \in  \Phi$, define $H_{\alpha} :=\{x\in V \mid(\alpha,x)=0\}.$ 
 For $\Sigma\subseteq\Phi^+$, the \textbf{Weyl subarrangement} $\A_{\Sigma}$ is  defined by $\A_{\Sigma}:= \{H_{\alpha} \mid \alpha\in\Sigma\}$. 
In particular, $\A_{\Phi^+}$ is called the \textbf{Weyl arrangement}. 

We can make $\Phi^+$ into a \emph{poset} (partially ordered set) by defining a partial order $\le$ on $\Phi^+$ as follows: $\beta_1 \le \beta_2$ if $\beta_2-\beta_1 \in\sum_{i=1}^\ell \Z_{\ge 0}\alpha_i$. 
The poset $(\Phi^+, \le)$ is called the \tbf{root poset} of $\Phi$. 
For an \emph{ideal} $I$  (Definition \ref{def:ideal}) of the root poset $\Phi^+$, the corresponding Weyl subarrangement $\A_{I}$ is called the \textbf{ideal subarrangement}.

\begin{theorem}[{\cite[Theorem 1.1]{ABCHT16}}]
\label{thm:ideal-free}
Any ideal subarrangement $\A_{I}$ is MAT-free, hence free.
\end{theorem}

The ideal subarrangements form a significant subclass of MAT-free arrangements.
However, there are many MAT-free arrangements (or MAT-partitions of a given MAT-free arrangement) that do not arise from ideal subarrangements (Example \ref{ex:C-vine}). 
One may wonder if the hyperplanes in an arbitrary MAT-free arrangement satisfy some poset structure similar to the root poset? 
This question was asked by Cuntz-M{\"u}cksch  \cite{CM20} and is the main motivation of our work.
\begin{question}[{\cite[Problem 47]{CM20}}]
\label{ques:CM}
Given an MAT-free arrangement $\A$, can we characterize all possible MAT-partitions of $\A$ by a poset structure generalizing the classical root poset? 
\end{question}

Cuntz-M{\"u}cksch's question is difficult in general as the number of different MAT-partitions of a given MAT-free arrangement  might be very large. 
Also, the definition of an MAT-partition itself does not reveal a natural choice of the desirable partial order. 
In the present paper, we pursue this question along \emph{graphic arrangements}, a well-behaved class of arrangements in which both freeness and MAT-freeness are completely characterized by combinatorial properties of graphs.

Let $G$ be a simple graph (i.e.~no loops and no multiple edges) with vertex (or node) set $N_G = \{v_{1}, \dots, v_{\ell}\}$ and edge set $E_G$.
The  \textbf{graphic arrangement} $\A_G$ is an arrangement in an $\ell$-dimensional vector space $V$ defined by
$$\A_G:= \{x_i - x_j=0 \mid \{v_i,v_j\} \in E_G \}.$$

A graph is  \textbf{chordal} if it does not  contain an induced \emph{cycle} of length greater than three. 
A chordal graph  is  \textbf{strongly chordal} if it  does not  contain a \emph{sun graph} as an induced subgraph (Definition \ref{def:strongly-chordal}). 
\begin{theorem}[\cite{St72}, {\cite[Theorem 3.3]{ER94}}]
\label{thm:free-chordal}
The graphic arrangement $\A_G$ is free if and only if $G$ is chordal.
\end{theorem}

 \begin{theorem}[{\cite[Theorem 2.10]{TT23}}]
\label{thm:MAT=SC}
The graphic arrangement $\A_G$ is MAT-free if and only if $G$ is strongly chordal.
\end{theorem}

While the definition of an MAT-free arrangement may seem technical at first glance, Theorem \ref{thm:MAT=SC} enables us to view MAT-freeness as a rather natural property. 
Furthermore, the correspondence between MAT-freeness and strong chordality establishes a nice analog\footnote{Many important concepts in the classical theory such as simplicial vertex and perfect elimination ordering of chordal graphs have their analogs in MAT-labeled graphs (see Remark \ref{rem:analog}).} of the classical correspondence between freeness and chordality. 

The good thing about graphs is that MAT-partition of a graphic arrangement can be rephrased in terms of a special edge-labeling of graphs, the so-called \emph{MAT-labeling} (Definition \ref{definition MAT-labeling}). 
A graph together with such a labeling is called an \tbf{MAT-labeled graph}.
To approach Question \ref{ques:CM} for graphic arrangements, the first question would be  how many non-isomorphic MAT-labelings can a (strongly chordal) graph have? 
A computation aided by computer for complete graphs on up to $8$ vertices gives us the sequence $1,1  ,1 , 2 , 6,40, 560, 17024$. 
Surprisingly, we found out that this sequence coincides with the number of equivalence classes of \emph{(graphical) regular vines} (or  \emph{R-vines}) in dimension up to $8$ given in   \cite[\S 10.3]{KJ11}.  
This observation is indeed compelling as it leads us to the notion of the \emph{node poset}   of a \emph{graphical vine} (Definitions \ref{def:vine} and \ref{def:node poset}), which is a perfect candidate for the poset structure we are looking for.

%******************************************************************************** 
\subsection{Main result}
\label{subsec:main-result}

In this paper, we first introduce a poset realization of graphical (R-)vines  (Definitions \ref{def-prop:V-poset} and \ref{def:RV-poset}). 
Our aim is to convert the important terms and properties of graphical vines into the language of posets in which considerable power and development of poset theory would be brought to bear.
We also introduce the notion of a \emph{locally regular vine} (or  \emph{LR-vine}) (Definition \ref {def:LRV-poset}). 
Roughly speaking, an LR-vine is a vine that ``locally" looks like an R-vine. 
It is worth mentioning that any ideal of an R-vine, or  \emph{m-vine} (Definition \ref {def:m-vine}) gets characterized by an LR-vine (Theorem \ref{thm:LR-characterize}). 

Having introduced the concepts, we define the category $\mathsf{MG}$ of MAT-labeled graphs and the category $\mathsf{LRV}$ of LR-vines (Definitions \ref{def:category-MG} and \ref{def:category-LR}). 
Our main result is that these categories are equivalent (Theorem \ref{thm:1-to-1}). 
In particular, we obtain the equivalence between the category of MAT-labeled complete graphs and the category of R-vines (Corollary \ref{cor:1-to-1-restrict}). 
The correspondences are summarized in Table \ref{tab:concepts}. 

To prove the equivalence between $\mathsf{MG}$  and $\mathsf{LRV}$, we construct two functors $\Psi\colon\MG \longrightarrow \LRV $ and $\Omega \colon \LRV\longrightarrow  \MG$. 
The former amounts to constructing an LR-vine from a given MAT-labeled graph which will be presented in Definition \ref{def:G-to-P} and Theorem \ref{thm:MAT-to-poset}. 
The proof is direct and largely dependent upon the notion of \emph{MAT-perfect elimination ordering} (Definition \ref{def:MAT-PEO}) developed in an earlier work of the last two authors \cite{TT23}. 
The argument on the functor $\Omega$ is however more complicated. 
We need to show some new properties of LR-vines in \S\ref{subsec:newprop} before giving the construction in Definition \ref{def:P-to-G} and Theorem \ref{thm:poset-to-MAT}. 

It is known that graphic arrangements are equivalent to Weyl subarrangements when the root system is of type $A$. 
It would be interesting to extend our main result to the root systems of other types. 
However, it is quite challenging since a complete characterization of either MAT-freeness or freeness of Weyl subarrangements is unknown except for type $A$.

%******************************************************************************** 
\subsection{Applications}
\label{subsec:appl}
From the view point of category theory, the equivalence establishes a strong similarity between the categories and allows many properties and structures to be translated from one to the other. 
We obtain three main applications from LR-vines to MAT-labeled graphs. 
First, our main theorem (\ref{thm:1-to-1}) gives a new poset characterization of the MAT-free graphic arrangements compared with the characterization by strong chordality in Theorem \ref{thm:MAT=SC}. 
In particular, LR-vine is an answer for Question \ref{ques:CM} in the case of graphic arrangements (see \S\ref{subsubsec:characterize}). 
We find it interesting that although the class of MAT-free arrangements strictly contains that of ideal subarrangements in general, any MAT-free \emph{graphic} arrangement is characterized by being an ideal of a poset structure. 
Second, an explicit formula for the number of non-isomorphic MAT-labelings of complete graphs is obtained as it equals the number of equivalence classes of regular vines (see \S\ref{subsubsec:MATCG}). 
Third, the notion of \emph{upper truncation} (Definition \ref{def:truncation}) of an LR-vine gives rise to a nontrivial graph operation which produces a new MAT-labeled graph from a given one (see \S\ref{subsubsec:UT}).

A vine is a graphical tool for representing the joint distribution of random variables. 
The first construction of a vine was given by Joe \cite{Joe94}, and the formal definition was given and refined further by Cooke, Bedford and Kurowicka \cite{Coo97, BC02, KC03}.
Vines have been studied extensively and proved to have various applications in probability theory and related areas. 
We refer the reader to \cite{KJ11} for a comprehensive handbook of vines. 
Our main result gives a new appearance and applications of vines in the arrangement theory. 
In the present paper, we do not pursue the probabilistic or applied aspects of vines (neither does the proof of the main result) but emphasize and develop more on the theoretical aspects. 
There are several new combinatorial properties of vines presented throughout, hoping that they will be useful for the future research on vines.  
For instance, we give an alternative way to associate an m-vine to a strongly chordal graph compared with the work of Zhu-Kurowicka \cite{ZK22} (see \S\ref{subsubsec:SC-mvine}), an extension of the notion of  \emph{sampling order} \cite{CKW15} from R-vine to LR-vine  (see \S\ref{subsubsec:SO}), and a conjectural formula for the number of ideals in a C-vine (see \S\ref{subsubsec:D-vine}).

\vskip .5em
\begin{table}[htbp]
\centering
{\renewcommand\arraystretch{1.5} 
\begin{tabular}{c|c}
Vine theory& Graph theory \\
\hline\hline
   R-vine & $\begin{array}{c}
     \text{MAT-labeled complete graph} 
    \end{array}$
     \\
\hline
$\begin{array}{c}
     \text{LR-vine}   \\
  \text{($=$ m-vine, or ideal of R-vine)} 
    \end{array}$
          & $\begin{array}{c}
     \text{MAT-labeled graph}   \\
  \text{($=$ MAT-free graphic arrangement)} 
    \end{array}$
\end{tabular}
}
\bigskip
\caption{\small Correspondence between concepts in vine and graph theories.}
\label{tab:concepts}
\end{table}

%********************************************************************************************************
\section{MAT-labelings of graphs}
\label{sec:MAT-labeling of graphs}
%******************************************************************************** 
\subsection{Graphs}
\label{subsec:graphs}
In this subsection, we recall some basic definitions and notions of graphs. 
All graphs in this paper are undirected, finite and simple.
Let $ G = (N_{G}, E_{G}) $ be a graph with the set $N_G$ of vertices (or nodes) and the set $E_G$ of edges (unordered pairs of vertices). 
In this paper, a vertex and a node  in a graph are synonyms. 
The former will be used more often for graphs, while the latter will be used for an element in a poset.

For $S \subseteq N_{G}$, denote by $G[S] = (S, E_{G[S]})$ where $E_{G[S]} =\{ \{u,v\} \in E_G \mid u, v \in S\}$ the \textbf{(vertex-)induced subgraph} of $S$. 
Denote by $ G\setminus S$ the induced subgraph $G[N_G\setminus S]$. 
In particular, $G\setminus  v:=G \setminus\{v\}$  when $v$ is a vertex of $G$. 
For $ F \subseteq E_{G} $, define the subgraph $ G\setminus F := (N_{G}, E_{G} \setminus F) $. 
In particular, $G\setminus e:=G \setminus \{e\} $  when $e$ is an edge of $G$.

A \textbf{complete graph} $K_{n} $ ($n \ge 0$) is a graph with vertex set $N_{K_{n}} = \{u_{1}, \dots, u_{n}\} $ and edge set 
\begin{align*}
E_{K_{n}} = \{\{u_{i},u_{j}\} \mid 1 \leq i < j \leq n\}. 
\end{align*}

In other words, a complete graph is a graph such that each pair of vertices is connected by an edge. 
Here the \tbf{order-zero graph} $K_0$, i.e.~ the graph having no vertices is also considered to be a complete graph.
A \textbf{clique} $C$ of $G$ is a subset of $N_{G}$ such that the induced subgraph $G[C]$ is a complete graph. 

For each $v \in N_{G}$, the \textbf{neighborhood} $ \nbd_{G}(v) $ of $v$ in $G$ is defined by $ \nbd_{G}(v) :=\{ u \in N_{G} \mid \{u,v\} \in E_G\}$. 
The \tbf{degree}  of $ v $ in $ G $ is defined by $ \deg_{G}(v) :=|\nbd_{G}(v) |$. 
A \tbf{leaf} is a vertex of degree $1$. 
A vertex is called \textbf{simplicial} if  its neighborhood is a clique. 
An ordering $ (v_{1}, \dots, v_{\ell}) $ of vertices of a graph $ G $ is called a \textbf{perfect elimination ordering (PEO)}  if $ v_{i} $ is simplicial in the induced subgraph $ G[\{v_{1}, \dots, v_{i}\}] $ for each $1 \le  i \le \ell $.

A \textbf{maximal clique} is a clique that it is not a subset of any other clique. 
A \textbf{largest clique} is a clique that has the largest possible number of vertices. 
Denote by $\mcK(G)$ the set of all maximal cliques of $G$. 
In particular, $ |\mcK(G)| =0$ or $1 $ if and only if $G$ is a complete graph. 
The \textbf{clique number} of $G$, denoted $\omega(G)$, is the number of vertices in a largest clique of $ G $.

A \tbf{walk} $W$ in a graph $G$ is a sequence of edges $(e_1, e_2,\ldots, e_{p})$ of $G$ for which there is a sequence of vertices $(v_1, v_2, \ldots, v_{p+1})$ of $G$ such that  $e_i = \{v_i, v_{i + 1}\}$ for $1 \le i \le p$. 
The vertices $v_1$ and $v_{p+1} $ are said to be \tbf{connected} by the walk $W$, called the \tbf{initial} and \tbf{final vertices} of $W$, respectively. 
The \textbf{length} of a walk is the number of edges in the walk (hence the length of $W$ is $p \ge 0$).
Throughout the paper,  a walk $W$ is denoted by its \tbf{vertex sequence} $W=(v_1, v_2, \ldots, v_{p+1})$. 
If $W_1=(v_1, v_2, \ldots, v_{p+1})$ and $W_2=(v_{p+1}, v_{p+2},\ldots, v_{n})$ are walks, then  the \tbf{concatenation} of $W_1$ and $W_2$ is the walk $(v_1, v_2, \ldots, v_{n})$.

A \tbf{path} $P=(v_1, v_2, \ldots, v_{p+1})$ is a walk with no repeated vertices except possibly the initial and final vertex. 
A \tbf{subpath} of $P$ is a path of the form $(v_i, v_{i+1}, \ldots, v_{j}) $ for some $1 \le i \le j \le p+1$.
The following fact is well-known.
	\begin{lemma} 
	\label{lem:walk-path} 
	Given two vertices $a,b $ in a graph $G$, every walk connecting $a$ and $b $ contains a path connecting $a$ and $b $. 
	\end{lemma}
	
	A graph $G$ is called  \tbf{connected} if any two vertices are connected by a walk (hence by a path by Lemma \ref{lem:walk-path} above).
 
A \textbf{$p$-cycle} $C_{p}=(v_1, v_2, \ldots, v_{p})$ for $p\ge3$ is a path with $v_{p}=v_1$. 
The $3$-cycle is also called a \textbf{triangle}.
A \textbf{chord} of a cycle is an edge connecting two non-adjacent vertices of the cycle.
  A \tbf{forest} is a graph containing no cycles. A \tbf{tree} is a connected forest. 
  In a forest (resp.~tree), any two distinct vertices are connected by at most (resp.~exactly) one path.

An \textbf{$n$-sun} $ S_{n} $ ($n \ge 3$) is a graph with vertex set $N_{S_{n}} = \{u_{1}, \dots, u_{n}\} \cup \{v_{1}, \dots, v_{n}\} $ and edge set 
\begin{align*}
E_{S_{n}} = \{\{u_{i},u_{j}\} \mid 1 \leq i < j \leq n\} \cup \{\{v_{i}, u_{j}\} \mid 1 \leq i \leq n, j \in \{i, i+1\} \}, 
\end{align*}
where we let $ u_{n+1} = u_{1} $.

\begin{definition}[Strongly chordal graph]
\label{def:strongly-chordal}
 A simple graph is  \textbf{strongly chordal} if it is $C_p$-free\footnote{In general, a graph is called $H$-free if it does not contain $H$ as an induced subgraph. It is not to be confused with ``MAT-free arrangement".} for $p\ge4$ and $S_n$-free for $n\ge3$. 
\end{definition}

The following property of a strongly chordal graph is a special case of \cite[Lemma 5.9]{TT23}.
	\begin{lemma} 
	\label{lem:max} 
 Let $ G $ be a strongly chordal graph with $ |\mcK(G)| \geq 2 $.
Then there exist distinct maximal cliques $ X_{0}, Y_{0} \in \mcK(G) $ such that $ X_{0} \cap Y_{0} \supseteq X_{0} \cap Y $ for all $ Y \in \mcK(G)  \setminus \{X_{0}\} $. 

	\end{lemma}
 
%******************************************************************************** 
 
\subsection{MAT-labeled graphs}
\label{subsec:MAT-labeled graphs}
In this subsection, we recall some preliminary definitions and facts of MAT-labeled graphs following \cite{TT23}.
An \textbf{edge-labeled graph} is pair $ (G,\lambda) $ where $ G $ is a simple graph and  $ \lambda \colon E_{G} \longrightarrow \mathbb{Z}_{>0} $ is a map, called \textbf{(edge-)labeling}. 
The following definition of an MAT-labeling is equivalent to the original one in \cite[Definition 4.2]{TT23}.

\begin{definition}[MAT-labeling]
\label{definition MAT-labeling}
Let $ (G,\lambda) $ be an  edge-labeled graph. 
For $ k \in \mathbb{Z}_{>0} $, let $ \pi_{k}:= \lambda^{-1}(k) \subseteq E_{G}$ denote the set of edges of label $k$. 
Define $ \pi_{\le k} := \pi_{1} \sqcup \dots \sqcup \pi_{k} $  and $ \pi_{<1}:= \varnothing $. 
The labeling $ \lambda $ is an \textbf{MAT-labeling} if the following two conditions hold for every $ k \in \mathbb{Z}_{>0} $. 
\begin{enumerate}[(ML1)]
\item\label{definition MAT-labeling cycle} Any edge  $e \in \pi_{\le k} $ does not form a cycle with edges in $\pi_k$.
\item\label{definition MAT-labeling triangle} Every edge  $e \in \pi_{k} $ forms exactly $ k-1 $ triangles with edges in $ \pi_{<k} $. 
\end{enumerate}
 
\end{definition}

 Given an edge  $e \in \pi_{k} $, a \tbf{conditioning vertex} of $e$ is a vertex that together with the endvertices of $e$ forms two edges both of label $<k$. 
 Condition (ML\ref{definition MAT-labeling triangle}) above can be rephrased as every edge  $e $ of label $k$ has exactly $k-1$ conditioning vertices.

\begin{definition}[MAT-labeled (complete) graph]
\label{def:MAT-labeled graph}
An edge-labeled graph $ (G,\lambda) $ is an \tbf{MAT-labeled graph} if $ \lambda $ is an  MAT-labeling of $G$. 
In particular, an MAT-labeled graph $ (G,\lambda) $ is an \tbf{MAT-labeled complete graph} if $G$ is a complete graph.
\end{definition}

MAT-partition of a graphic arrangement is nothing but MAT-labeling of the underlying graph
\cite[Proposition 4.3]{TT23}. 
Thus, MAT-free graphic arrangement and MAT-labeled graph are essentially the same object.
The following properties of an MAT-labeled graph are deduced thanks to the relation with freeness.

\begin{lemma}[{\cite[Proposition 4.8]{TT23}}]
\label{lem:largest-label}
 Let $(G,\lambda)$  be an MAT-labeled graph with clique number $ \omega(G)$. 
 Then the following statements hold. 
\begin{enumerate}[(1)]
\item The   largest label of edges in $(G,\lambda)$ is equal to $ \omega(G)-1 $. 
\item There exists a bijection between the set of largest cliques of $ G $ and $ \pi_{n}$  where $n = \omega(G)-1$ via the relation: For each $ e \in  \pi_{n}$, there exists a unique largest clique of $G$ containing the endvertices of $e$. 
\end{enumerate}
\end{lemma}

\begin{lemma}[{\cite[Proposition 4.4]{TT23}}]
\label{lem:card-complete}
If $ \lambda $ is an MAT-labeling of the complete graph $ K_{\ell} $, then $ |\pi_{k}| = \ell-k $ for all $1 \le k \le \ell-1$. 
\end{lemma}

Now we present some results on the restrictions of MAT-labelings.

\begin{definition}[MAT-labeled subgraph]
\label{def:MAT-labeled subgraph}
Let $(G,\lambda)$  be an MAT-labeled graph. 
An edge-labeled graph $(G', \lambda')$ is an \tbf{MAT-labeled subgraph} of $(G, \lambda)$, written $(G', \lambda') \le (G,\lambda)$, if $G'$ is a subgraph of $G$,  $\lambda' = \lambda|_{E_{G'}}$ i.e.~$\lambda' $ is the restriction of $\lambda$ on $E_{G'}$, and $(G', \lambda')$ itself is an MAT-labeled graph.
 
\end{definition}

\begin{lemma}[{\cite[Lemma 4.7]{TT23}}]
\label{lem:restriction to union}
Let $(G,\lambda)$  be an MAT-labeled graph and let $ F_{1}, F_{2} \subseteq E_{G} $. 
If $ \lambda|_{F_{1}} $ and $ \lambda|_{F_{2}} $ are MAT-labelings of the subgraphs $(N_{G}, F_1)$ and $(N_{G}, F_2)$, respectively, then $ \lambda|_{F_{1} \cup F_{2}} $ is an MAT-labeling of $(N_{G}, F_{1} \cup F_{2})$. 
\end{lemma}

\begin{lemma}[{\cite[Lemma 4.9]{TT23}}]
\label{lem:restriction to intersection of maximal cliques}
Let $(G,\lambda)$  be an MAT-labeled graph. 
Let $X$ denote the intersection of some maximal cliques of $G$. 
Then $(G[X],  \lambda|_{E_{G[X]}} ) \le (G, \lambda)$.
\end{lemma}

The following is an immediate consequence of Lemmas \ref{lem:restriction to union} and \ref{lem:restriction to intersection of maximal cliques} above.

\begin{corollary}
\label{cor:restriction}
 Let $ (G, \lambda) $ be an MAT-labeled graph. 
 Let $\B \subseteq \mcK(G) $ be a set of some maximal cliques of $G$. 
 Let $ G_\B$ be the subgraph of $ G $ with vertex set $ N_{G_\B}:= \cup_{Y \in \B}Y$ and edge set $ E_{G_\B}:=  \cup_{Y \in \B} E_{G[Y]}$. 
 Then  $(G_\B, \lambda|_{E_{G_\B}}) \le (G, \lambda) $. 
\end{corollary}

 \begin{notation}
\label{nota:lambda} 
For simplicity of notation,  if $\lambda \colon E_{G}  \longrightarrow \mathbb{Z}_{>0} $ is a labeling and $\{u, v\} \in E_G$ is an edge, we write $\lambda(u,v):=\lambda(\{u,v\})$ for the label of $\{u, v\} $.
\end{notation} 

The following analogs of simplicial vertex and \emph{perfect elimination ordering} of chordal graphs are important concepts in the study of MAT-labeled graphs.

\begin{definition}[MAT-simplicial vertex]
\label{definition MAT-simplicial}
Given an edge-labeled graph $ (G,\lambda) $, a vertex $ v \in N_{G} $ is  \textbf{MAT-simplicial} if the following conditions hold. 
\begin{enumerate}[(MS1)]
\item\label{definition MAT-simplicial 1} $ v $ is a simplicial vertex of $ G $.
\item\label{definition MAT-simplicial 2} The edges of $G$ incident on $v$ are labeled by labels from $1$ to $\deg_{G}(v)$, i.e.~$ \{\lambda(u,v) \in \mathbb{Z}_{>0} \mid u \in \nbd_{G}(v) \} = \{1,2, \dots, \deg_{G}(v)\} $. 
\item\label{definition MAT-simplicial 3} For any distinct vertices $ u_{1}, u_{2} \in \nbd_{G}(v) $, $ \lambda(u_{1}, u_{2}) < \max\{\lambda(u_{1}, v), \lambda(u_{2}, v)\} $. 
\end{enumerate}
\end{definition}

\begin{lemma}[{\cite[Lemma 5.2]{TT23}}]
\label{lem:MATS-existence}
If $ (G, \lambda) $ is an MAT-labeled complete graph, then the endvertices of the edge of largest label are MAT-simplicial. 
\end{lemma}

\begin{lemma}[{\cite[Proposition 5.3]{TT23}}]
\label{lem:MAT-simplicial}
Let $ (G, \lambda) $ be an edge-labeled graph having at least $2$ vertices. 
Suppose that $ v  $ is an MAT-simplicial vertex of $ (G,\lambda) $. 
Then 
  $ \lambda $ is an MAT-labeling of $ G $ if and only if 
 $ \lambda|_{E_{G\setminus v}} $ is an MAT-labeling of $ G\setminus v $. 
 
\end{lemma}

\begin{definition}[MAT-PEO]
\label{def:MAT-PEO}
Let $ (G,\lambda) $ be an edge-labeled graph  on $\ell$ vertices. 
An ordering $ (v_{1}, \dots, v_{\ell}) $ of vertices in $ G $ is an \textbf{MAT-perfect elimination ordering (MAT-PEO)} of $ (G,\lambda) $ if $ v_{i} $ is MAT-simplicial in $ (G_{i}, \lambda_{i}) $ for each $1 \le  i \le \ell $, where $ G_{i}:= G[\{v_{1}, \dots, v_{i}\}] $ and $ \lambda_{i}:= \lambda|_{E_{G_{i}}} $. 
\end{definition}

\begin{theorem}[{\cite[Theorem 5.5]{TT23}}]
\label{thm:MAT-PEO}
An edge-labeled graph $ (G, \lambda) $  is an MAT-labeled graph if and only if there exists an MAT-PEO of $ (G, \lambda) $. 
 
\end{theorem}

\begin{remark} 
\label{rem:analog} 
It is known that a graph is chordal if and only if it has a perfect elimination ordering \cite{FG65}. 
Theorem \ref{thm:MAT-PEO} is an analog of this classical result for MAT-labeled graphs.
\end{remark}

The method of merging regular vines was given in \cite{CKW15, ZK22}. 
We have a very similar\footnote{These methods are actually equivalent in the sense that they produce the same output as a consequence of our Corollary \ref{cor:1-to-1-restrict}.} method for merging MAT-labeled complete graphs.

\begin{lemma}[Merging MAT-labeled complete graphs {\cite[Lemma 5.7]{TT23}}]
\label{lem:merge}
Let $ (G_1, \lambda_1) $ and $ (G_2, \lambda_2) $ be MAT-labeled complete graphs. 
Suppose $G_1[N_{G_1}  \cap N_{G_2}] = G_2[N_{G_1}  \cap N_{G_2}]$ and denote this common complete graph by $G'$. 
Assume that there exists an MAT-labeling $\lambda'$ of $G'$ such that $(G',\lambda')\le (G_1, \lambda_1) $ and $(G',\lambda')\le  (G_2, \lambda_2) $. 
Let $G$ denote the complete graph with vertex set $N_{G_1}  \cup N_{G_2}$. 
Then there exists an MAT-labeling $ \lambda $ of $ G $ such that   $ (G_1, \lambda_1) \le (G, \lambda) $ and $ (G_2, \lambda_2) \le (G, \lambda) $.
\end{lemma}

\begin{proposition} 
\label{prop:extend}
Let $ (G, \lambda) $ be an MAT-labeled graph and $K$ be the complete graph with vertex set $N_G$. 
Then there exists an MAT-labeling $\widetilde\lambda$ of $K$  such that $ (G, \lambda) \le (K,\widetilde\lambda) $.

\end{proposition}

   \begin{proof}
     We argue by induction on the number $ |\mcK(G)| $ of maximal cliques of $G$. 
     If $ |\mcK(G)| =0$ or $1 $, then the assertion follows since $G$ is a complete graph. 
   Suppose  $ |\mcK(G)| \geq 2 $.   

Note that by Theorem \ref{thm:MAT=SC}, $G$ is a strongly chordal graph.
By Lemma \ref{lem:max}, there exist distinct $ X_{0}, Y_{0} \in \mcK(G) $ such that $ X_{0} \cap Y_{0} \supseteq X_{0} \cap Y $ for all $ Y \in \mcK(G)  \setminus \{X_{0}\} $. 
Denote $\B := \mcK(G) \setminus\{X_{0}\} $. 
Let $ G_\B$ be the subgraph of $ G $ with vertex set $ N_{G_\B}:= \cup_{Y \in \B}Y$ and edge set $ E_{G_\B}:=  \cup_{Y \in \B} E_{G[Y]}$. 
Thus $ G_\B$ has $ |\mcK(G)| -1$ maximal cliques. 
Moreover, $(G_\B, \lambda|_{E_{G_\B}})$ and $(G[X_0], \lambda|_{E_{G[X_0]}})$ are MAT-labeled graphs by Corollary \ref{cor:restriction}. 
By the induction hypothesis, there exists an MAT-labeling $ \lambda_\B$ of the complete graph $K_\B$ with vertex set $N_{G_\B}$ such that $ (G_\B, \lambda|_{E_{G_\B}}) \le (K_\B,\lambda_\B)  $.

Now we are in the setting of Lemma \ref{lem:merge} with $ (G_1, \lambda_1)=(G[X_0], \lambda|_{E_{G[X_0]}})$ and $ (G_2, \lambda_2) =(K_\B,\lambda_\B) $. 
Indeed, first note that
\begin{align*}
N_{G_1}  \cap N_{G_2} = 
X_{0} \cap \left(\bigcup_{Y \in \B}Y\right)
= \bigcup_{Y \in \B}(X_{0} \cap Y) = X_{0} \cap Y_{0}.
\end{align*}
Hence $G_1[N_{G_1}  \cap N_{G_2}] = G_2[N_{G_1}  \cap N_{G_2}] =G[X_{0} \cap Y_{0}]$. 
Denote this common complete graph by $G'$. 
 By Lemma \ref{lem:restriction to intersection of maximal cliques}, the restriction $\lambda':=\lambda|_{E_{G'}}$ is an MAT-labeling $G'$. 
 Hence $(G',\lambda')\le (G_1, \lambda_1) $ and $(G',\lambda')\le  (G_2, \lambda_2) $.
 
Therefore, by Lemma \ref{lem:merge}, there exists an MAT-labeling $ \widetilde\lambda $ of the complete graph $K$ with vertex set $N_{G_1}  \cup N_{G_2}=N_{G}$ such that   $ (G_1, \lambda_1) \le(K,\widetilde\lambda)$ and $ (G_2, \lambda_2) \le (K,\widetilde\lambda)$. 
In particular, $ (G, \lambda) \le (K,\widetilde\lambda) $.
 
\end{proof}

The following ``gluing trick"  plays an important role in the construction of an MAT-labeling for a given strongly chordal graph in \cite[\S5.2]{TT23}.

\begin{lemma}[Gluing MAT-labeled graphs  {\cite[Theorem 5.8]{TT23}}]
\label{lem:glue}
Let $ (G_1, \lambda_1) $ and $ (G_2, \lambda_2) $ be MAT-labeled graphs such that $G_1[N_{G_1}  \cap N_{G_2}] = G_2[N_{G_1}  \cap N_{G_2}]$. 
Suppose that this common subgraph, denoted $G'$, is a complete graph. 
Suppose further that there exists an MAT-labeling $ \lambda' $ of $ G' $ such that   $ (G', \lambda')   \le(G_1, \lambda_1)$ and $ (G', \lambda')   \le (G_2, \lambda_2) $.
Define an edge-labeled graph $ (G, \lambda) $ by $ N_{G} = N_{G_1}  \cup N_{G_2}$, $E_{G} = E_{G_1} \cup E_{G_2} $, $\lambda|_{E_{G_1}} = \lambda_{1}, \lambda|_{E_{G_2}} = \lambda_{2}$.
Then $ (G, \lambda) $ is an MAT-labeled graph.

\end{lemma}

  We close this subsection by introducing the notion of \emph{principal cliques} in an MAT-labeled graph. 
  This notion will be useful for the construction of an LR-vine from a given MAT-labeled graph in Definition \ref{def:G-to-P}.
		\begin{lemma}[Principal clique]
\label{lem:edge-clique}
Let $(G,\lambda)$ be an MAT-labeled graph. 
Let $e=\{i,j\} \in \pi_k$ be an edge in $G$ of label $k$ and $h_1, \ldots, h_{k-1}$ be the conditioning vertices of $e$. 
Then the set $K_e:=\{i,j, h_1, \ldots, h_{k-1}\}$ is a clique of $G$. 
Moreover, $(G[K_e],  \lambda|_{E_{G[K_e]}} ) \le (G, \lambda)$ and all the edges in $G[K_e] \setminus e$ have label $<k$.
We call   $K_e$ the \tbf{principal clique generated by $e$}.
\end{lemma}

		\begin{proof}
		Let $G'$ denote the graph obtained from $G$ by removing all edges of labels $>k$. 
By definition, $(G', \lambda') \le (G,\lambda)$ where $\lambda' := \lambda|_{E_{G'}}$.
Since $e$ is an edge of largest label in $(G', \lambda')$, by Lemma \ref{lem:largest-label}, there exists a unique largest clique $C$ of $G'$ containing  the endvertices of $e$. 
Note also that $C$ does not contain the endvertices of any edge of label $k$ apart from $e$. 
Moreover, $(G[C],  \lambda|_{E_{G[C]}} ) \le (G', \lambda') $ by
Lemma  \ref{lem:restriction to intersection of maximal cliques}. 
Thus the number of conditioning vertices of $e$ is $k-1$ in both $G$ and $G[C]$.
Hence $C=K_e$. 
	\end{proof}
	
	The converse of  Lemma \ref{lem:edge-clique} is also true.
	
			\begin{lemma}
\label{lem:max-prin}
If $C$ is a clique in an MAT-labeled graph  $ (G, \lambda) $ such that  $(G[C],  \lambda|_{E_{G[C]}} ) \le (G, \lambda) $, then $C$ is a principal clique.
In particular, any maximal clique is principal.
\end{lemma}
		\begin{proof}
		Let $G':= G[C]$ and $  \lambda':=\lambda|_{E_{G[C]}}$. 
 By the assumption, $(G', \lambda')$   is an MAT-labeled complete graph. 
 Thus $(G', \lambda')$  has a unique edge of maximal label by Lemma \ref{lem:card-complete}. 
Hence $C$ is a principal clique in $(G', \lambda')$ (generated by this unique edge) hence in $(G,\lambda)$. 
The consequence follows directly from Lemma  \ref{lem:restriction to intersection of maximal cliques}. 
 
	\end{proof}

			\begin{lemma}
\label{lem:pc-cover}
Let $(G,\lambda)$ be an MAT-labeled graph. 
Let $e=\{i,j\} $ be an edge in $G$ and $K_e$ the principal clique generated by $e$. 
Then both $C_1:=K_e\setminus\{i\}$ and $C_2:=K_e\setminus\{j\}$ are  principal cliques in $(G,\lambda)$. 
Moreover, if $C$ is a  principal clique in $(G,\lambda)$ such that $C \subsetneq K_e$, then either $C \subseteq C_1$ or $C \subseteq C_2$.
\end{lemma}

		\begin{proof}
Let $G':= G[K_e]$ and $  \lambda':=\lambda|_{E_{G[K_e]}}$. 
By Lemma \ref{lem:edge-clique}, $(G', \lambda')$ is an MAT-labeled complete graph in which $e$ is the unique edge of largest label. 
By Lemma \ref{lem:MATS-existence}, the endvertices $i$ and $j$ of $e$ are MAT-simplicial. 
By Lemma \ref{lem:MAT-simplicial}, $(G'[C_1], \lambda'|_{E_{G'[C_1]}})$ and $(G'[C_2], \lambda'|_{E_{G'[C_2]}})$ are MAT-labeled complete graphs. 
By Lemma \ref{lem:max-prin}, $C_1 $ and $C_2 $ are  principal cliques in $(G,\lambda)$. 

Let $C$ be a  principal clique in $(G,\lambda)$ such that $C \subsetneq K_e$. 
Then $C = K_f$ for some edge $f$ in $G[K_e]$. 
Moreover,  $ \lambda(f) = |C| -1 < |K_e| - 1 = \lambda(e).$
Thus $e$ is not an edge of $G[C]$ by Lemma \ref{lem:edge-clique}.
Hence $C$ cannot contain both $i$ and $j$.
It follows that either $C \subseteq C_1$ or $C \subseteq C_2$.

	\end{proof}

 %********************************************************************************************************
\section{Vines: graphical and poset definitions}
%******************************************************************************** 
\subsection{Posets}
\label{subsec:posets}
In this subsection, we recall some basic definitions and notions of posets. 
All posets $\Pc=(\Pc, \le_{\Pc})$  in this paper are finite.
For a poset $\Pc$, an element $a \in\Pc$ is called \tbf{maximal} (resp.~\tbf{minimal}) if there is no other element $b \in\Pc$ such that $a < b$ (resp.~ $a > b$). 
Denote by  $\max(\Pc)$ (resp.~$\min(\Pc)$) the set of all maximal (resp.~minimal) elements in $\Pc$.

\begin{definition}[Join] 
\label{def:join}
 Let $\Pc$ be a poset and let $x,y \in \Pc$. An element $v \in \Pc$ is called the \tbf{join} of $x$ and $y$, denoted $ x\vee y$, if the following two conditions are satisfied:
 \begin{enumerate}[(1)]
		\item  $x \le v$ and $y \le v$.
		\item For any $w \in \Pc$, if $x \le w$ and $y \le w$, then $v \le w$.
	\end{enumerate}
The join $ x\vee y$ is unique if it exists.
\end{definition}

\begin{definition}[Induced subposet] 
\label{def:ind-subposet}
A poset $(\Qc,\le_\Qc)$ is an \textbf{induced subposet} of a poset $(\Pc,\le_\Pc)$ if $\Qc\subseteq \Pc$ and for any $a,b \in \Qc$ it holds that $a \le_\Qc b$ if and only if $a \le_\Pc b.$
\end{definition}

For $x, y \in \Pc$, by $y$ \tbf{covers} $x$, we mean $x<y$ and $x\le z<y$ implies $x = z$.

\begin{definition}[Graded poset] 
\label{def:gposet}
A finite poset $\Pc$ is \tbf{graded} if there exists a \tbf{rank function} $\rk=\rk_\Pc: \Pc  \longrightarrow \Z_{\ge0}$ satisfying the following three properties:
 \begin{enumerate}[(1)]
		\item For any $x,y \in \Pc$, if $x < y$ then $\rk(x) < \rk(y)$.
		\item If $y$ covers $x$, then $\rk(x) =\rk(y)-1$.
		\item All minimal elements of $\Pc$ have the same rank.
		In this paper, we assume\footnote{A motivation for this assumption is the equivalence between D-vine and root poset of type $A$ (Remark \ref{rem:D-vine-RP}). The latter is graded by heights of positive roots, and all the minimal elements (simple roots) have rank (height) $1$.}
		  $\rk(x)=1$ for all $x \in \min(\Pc)$.
	\end{enumerate}
Equivalently, for every $x \in \Pc$, all maximal chains among those with $x$ as greatest element have the same length. 	

The \tbf{dimension}\footnote{The term ``dimension" of a poset may have a different meaning in the other context. The present definition is to make a compatibility for dimensions of a vine (Remark \ref{rem:graded}) and the ambient space of graphic arrangements.}
 $\dim(\Pc)$ of $\Pc$ is defined as $\dim(\Pc):= |\min(\Pc)|$.
The  \tbf{rank} $\rk(\Pc)$ of a graded poset $\Pc$ with rank function $\rk$ is defined as
$$\rk(\Pc):=\max \{\rk(x) \mid x \in \Pc\}.$$
\end{definition}

\begin{definition}[Ideal, principal ideal] 
\label{def:ideal}
 Let $\Pc$ be a poset. 
An (order) \tbf{ideal} $\I$ of $\Pc$ is a downward-closed subset, i.e.~for every $x \in \Pc$ and $y \in \I$, $x \le y$ implies that $x  \in \I$. 
For $a \in \Pc$, the ideal
$$\Pc_{\le a} : = \{ x \in \Pc \mid x \le a\}$$
is called the \tbf{principal} ideal of $\Pc$ generated by $a$.
\end{definition}

\begin{definition}[Poset homomorphism] 
\label{def:poset-iso}
	Let $\Pc$ and $ \Pc'$ be posets. 
A  \tbf{(poset) homomorphism} $\varphi : \Pc \longrightarrow \Pc'$ is an order-preserving map, i.e.~$x \le y$ implies $\varphi(x) \le \varphi(y)$ for all  $x, y \in \Pc$. 

We call $\varphi$ a  \tbf{join-preserving} homomorphism if for any $x, y \in \Pc$ such that the join $ x\vee y$ exists, then $ \varphi (x) \vee \varphi(y)$ exists and $\varphi( x\vee y) = \varphi (x) \vee \varphi(y)$.

We call $\varphi$ an  \tbf{isomorphism} if $\varphi$ is bijective and its inverse is a homomorphism. 
The posets $\Pc$ and $ \Pc'$ are said to be  \tbf{isomorphic}, written $\Pc \simeq \Pc'$ if there exists an isomorphism  $\varphi : \Pc \longrightarrow\Pc'$. 

When $\Pc=(\Pc, \rk)$ and $\Pc'=(\Pc', \rk')$ are graded posets,   a homomorphism $\varphi : \Pc \longrightarrow \Pc'$  is called \tbf{rank-preserving} if $\rk'( \varphi (x) ) = \rk(x)$ for all $x \in \Pc $. 
\end{definition}

A rank-preserving homomorphism $\varphi : \Pc \longrightarrow \Pc'$ sends a minimal element to a minimal element, i.e.~$\varphi(\min(\Pc)) \subseteq \min(\Pc')$. 
Any isomorphism between two graded posets is a homomorphism preserving rank and join. 
 
%******************************************************************************** 
\subsection{Vines (graphical definition)}
\label{subsec:vines}
First we recall the graphical definition of a  \emph{vine}  following \cite[Definition 4.1]{BC02}. 

\begin{definition}[Graphical definition of vine] 
\label{def:vine}
Let $1\le n \le \ell$ be positive integers.
A \tbf{(graphical) vine} $\V$ on $\ell$ elements $[\ell]= \{1,\ldots,\ell\} $ (or more generally, on an $\ell$-element set called $N_1$) is an ordered $n$-tuple $\V=(F_1,F_2,\ldots,F_{n})$ such that
\begin{enumerate}[(1)]
		\item $F_1$ is a forest with nodes $N_1 = [\ell]$ and a set of edges denoted $E_1$,
		\item  for $2\le i \le n$, $F_i $  is a  forest  with nodes $N_i = E_{{i-1}}$ and edge set $E_i$. 
	\end{enumerate}
\end{definition} 
	  We call $F_i$ the \tbf{$i$-th associated forest} of $\V$. 
 A graphical vine is uniquely determined by its associated forests. 
	Denote by $N(\V)=N_1 \cup \cdots \cup N_n$   the set  of nodes (of associated forests) of $\V$. 
	We call the numbers $n$ and $\ell$ the \tbf{rank} and \tbf{dimension} of $\V$, respectively.

If node $u$ is an element of node $v$, i.e.~$u \in v$, we say that $u$ is a \tbf{child} of $v$.
If $v$ is reachable from $u$ via the membership relation: $u \in u_1 \in \cdots \in v$, we say that $u$ is a \tbf{descendant} of $v$.

\begin{definition}[Node poset] 
\label{def:node poset}
 Let $\V$ be a graphical vine with node set $N(\V)$. 
 The \tbf{node poset} $\Pc=\Pc(\V)$ of $\V$ is the poset $(N(\V), \le)$ defined as follows: For any $u,v \in N(\V)$, 
$$\mbox{$ u \le v$ \quad if  \quad $u$ is a descendant of $v$.}$$
\end{definition} 

\begin{remark} 
\label{rem:graded}
We emphasize that a graphical vine is uniquely determined by its node poset. 
The terminology ``rank" of a vine has motivation from poset theory. 
If a vine $\V$ is an ordered $n$-tuple, then $\Pc=\Pc(\V)$ is a graded poset with rank function $\rk(v) = i$ for $v \in N_i$ ($1 \le i \le n$). 
Thus this number $n$ equals the rank   of $\Pc$. 
In addition, the dimension of $\V$ equals the number of minimal elements in $\Pc$, or the dimension of $\Pc$.
\end{remark}

Now we introduce the notion of an \emph{induced subvine}, or more generally, a \emph{subvine} of a vine following \cite[\S5]{CKW15}, \cite[\S2.2.1]{ZK22}. 

\begin{definition}[Subvine, induced subvine] 
\label{def:subvine}
 Let $\V=(F_1,F_2,\ldots,F_{n})$ be a graphical vine. 
 
 \begin{enumerate}[1.]
\item
An ordered $p$-tuple $\V'=(F'_1,F'_2,\ldots,F'_{p})$ for $p \le n$ is called a \tbf{subvine} of $\V$ if $F'_i$ is a subgraph of $F_i$ for each $1\le i \le p$ and $\V'$ itself is a vine.
	
\item Given a subset $S \subseteq N_1$, define a vine $\V[S] = (F'_1,F'_2,\ldots,F'_{p})$ on the set $S$ as follows: 
 \begin{enumerate}[(1)]
		\item $F'_1 = F_1[S]$ with edge set $E'_1 \subseteq E_1 =N_2$,
		\item  for $2\le i \le p$, $F'_i = F_i[E'_{i-1}]$   with  edge set $E'_i \subseteq E_{i}=N_{i+1}$. 
	\end{enumerate}
	We call   $\V[S]$ the \tbf{subvine of $\V$ induced by $S$}.
	\end{enumerate}
 
\end{definition}

\begin{remark} 
\label{rem:Indsubv<subv} 
Any induced subvine is obviously a subvine but the converse is not necessarily true. 
For example, let $\V=(F_1,F_2)$ be a vine of dimension $2$ with $N_1=\{1,2\}$, $N_2=E_1 = \{ \{1,2\}\}$, $E_2=\varnothing$. 
The subvine $\V'=(F'_1)$ defined by $N'_1=\{1,2\}$, $E'_1=\varnothing$ is not an induced subvine of $\V$.
\end{remark}

%******************************************************************************** 
\subsection{Vines (poset definition)}
\label{subsec:V-posets}
\begin{assumption}
\label{aspn:V-poset} 
From now on, unless otherwise stated we assume that $\Pc$ is a finite graded poset with a rank function $\rk: \Pc  \longrightarrow \Z_{>0}$. 
Denote $n:=\rk(\Pc)$   and  $\ell :=\dim(\Pc)$. 
For $v\in \Pc$, denote by $\E(v)$ the set of elements covered by $v$. 
For $i \ge 0$, define $\Pc_i := \{ v \in \Pc \mid \rk(v) =i\}$ and $\E(\Pc_i) := \{ \E(v) \mid v \in \Pc_i \}$.
If $\Pc$ is an $\ell$-dimensional poset, we assume $\Pc_1=\min(\Pc)=[\ell]$.
\end{assumption} 

As noted earlier in Remark \ref{rem:graded},  we may think of a graphical vine and its node poset essentially as the same object. 
It is thus natural to look for a characterization of the node poset of a vine. 
We give below such a characterization obtained immediately from Definition \ref{def:vine}.

\begin{def-prop}[Poset definition of vine] 
\label{def-prop:V-poset} 
A finite graded poset  $\Pc$ is the node poset of a graphical vine if and only if $\Pc$  satisfies the following conditions: 
\begin{enumerate}[(1)]
		\item Every non-minimal node covers exactly two other nodes, and any two distinct nodes of the same rank are covered by at most one node. 
		\item For each $1\le i \le n = \rk(\Pc)$, the graph $F_i = (N_i, E_i)$   with node set $N_i:=\Pc_{i} $ and edge set $E_i:=\E(\Pc_{i+1}) $ is a  forest.
	\end{enumerate}
\end{def-prop}

\begin{assumption}
\label{aspn:vine-to-Vposet} 
From now on, unless otherwise stated, by a vine $\Pc$ we mean a finite graded poset satisfying the two conditions in \ref{def-prop:V-poset}. 
We will also retain the notion $i$-th associated forest $F_i =(\Pc_i, \E(\Pc_{i+1}))$ ($1\le i \le n$) of $\Pc$. 
If $v$ is a node in a vine $\Pc$ and $\E(v) =\{a,b\}$, we will often abuse notation and write $v=\{a,b\}$. 
This notation is compatible with the notation of node/edge in the graphical definition of a vine.
\end{assumption} 

The main reason why we choose the poset definition of a vine is because many terms and properties of a (graphical) vine have natural meanings in the language of posets. 
For example, subvine corresponds to ideal (Lemma \ref{lem:subvine=ideal}), conditioned set corresponds to join (Lemma \ref{lem:LR-paths}), and m-vine corresponds to LR-vine or local regularity of vine (Theorem \ref{thm:LR-characterize}).

Under this consideration, the following poset definition of a  \emph{regular vine} is equivalent to the well-known graphical definition of it in the literature, e.g.~\cite[Definition 4.1]{BC02}. 

\begin{definition}[R-vine] 
\label{def:RV-poset}
A vine $\Pc$  is a \tbf{regular vine}, or an \emph{R-vine} for short, if $\Pc$  satisfies the following conditions: \begin{enumerate}[(1)]
		\item $\rk(\Pc) =\dim(\Pc)$, i.e.~ $n=\ell$.
		\item Each associated forest $F_i =(\Pc_i, \E(\Pc_{i+1}))$  is a  tree ($1\le i \le n$).
		\item \tbf{Proximity}: For any distinct nodes $a,b \in \Pc_i$ for $i\ge2$, if $a$ and $b$ are covered by a common node, then $a$ and $b$ cover a common node.
	\end{enumerate}
 
\end{definition}

\begin{remark} 
\label{rem:card-R-vine} 
If $\Pc$ is an R-vine of rank $n$, then $|\Pc_i| = n+1-i$ for each $1 \le i \le n$. 
In particular, $|\Pc| = n(n+1)/2$.
\end{remark}

Next we introduce the notion of a \emph{locally}  regular vine.

\begin{definition}[LR-vine] 
\label{def:LRV-poset}
A vine $\Pc$  is a \tbf{locally regular vine}, or an \emph{LR-vine} for short, if every principal ideal of $\Pc$ is an R-vine.
\end{definition}

\begin{remark} 
\label{rem:ideal-LRvine} 
Intuitively, an LR-vine is a vine that ``locally" looks like an R-vine. 
In particular, any R-vine is an LR-vine (see Proposition \ref{prop:LR=P}). 
Any ideal of a vine (resp.~an LR-vine) is itself a vine (resp.~an LR-vine).
\end{remark}

\begin{lemma} 
\label{lem:subvine=ideal}
  Let $\V$ be a graphical vine with the node poset $\Pc(\V)$. 
  A subset $\I$ is an ideal of $\Pc(\V)$ if and only if $\I = \Pc(\V')$ where $\V'$ is a subvine of $\V$ uniquely determined by $\I$. 
  
  As a result, there is a one-to-one correspondence between the subvines of a graphical vine and the ideals of its node poset.
\end{lemma}

   \begin{proof}
   Let $\V'$ be a subvine of $\V$. 
   Since $\V'$ itself is a vine, if $v$ is a node in $\Pc(\V')$, then both children hence all descendants of $v$ are also nodes in $\Pc(\V')$. 
 Hence  $\Pc(\V')$ is an ideal of $\Pc(\V)$. 
   Conversely, let $\I$ be an ideal of the vine $\Pc(\V)$. 
By Remark \ref{rem:ideal-LRvine}, $\I$ itself is a vine.
By Remark \ref{rem:graded}, $\I$ uniquely determines  a vine $\V'$  which is a subvine of $\V$ and satisfies $\I = \Pc(\V')$. 

\end{proof}

We close this section by recalling the definition of an \emph{m-saturated vine} from \cite[Definition 4.2]{KC06}.
\begin{definition}[M-vine] 
\label{def:m-vine}
 A  vine $\Pc$ is called an \tbf{m-saturated vine}, or an \emph{m-vine} for short, if  $\Pc$ is an ideal of an R-vine.
\end{definition} 

By Remark \ref{rem:ideal-LRvine}, any m-vine is an LR-vine. 
We will see in Theorem \ref{thm:LR-characterize} that the converse also holds true.

  %********************************************************************************************************
 
\section{From MAT-labeled graphs to LR-vines}
\label{sec:MAT-to-vine}

 %********************************************************************************************************
\subsection{Some known properties of vines}
\label{subsec:known}

We begin by defining some statistics on the nodes of a vine. 
They play an important role in probabilistic applications of vines, e.g.~\cite[Theorem 3]{BC01}.

\begin{definition}[$k$-fold union, complete union] 
\label{def:union}
Let $\Pc$  be a vine of  rank $n$. 
For any node $v_i \in \Pc_i$ ($1 \le i \le n$) and integer $k$ with $0 \le k \le i-1$, the \tbf{$k$-fold union} of $v_i$  is the subset $U_{v_i} (k) \subseteq \Pc_{i-k}$ defined by 
$$U_{v_i} (k) :=\{ x \in \Pc_{i-k} \mid   x \le v_i \}.$$

The \tbf{complete union} $U_{v_i}$ of $v_i \in \Pc_i$ is defined as the $(i-1)$-fold union of $v_i$, i.e.~
$$U_{v_i} := U_{v_i} (i-1) \subseteq \Pc_1.$$
 
\end{definition}

\begin{definition}[Conditioned set, conditioning set] 
\label{def:ccc-sets}
Let $\Pc$  be a vine of  rank $n$. 
Let $v_i = \{a, b\}\in \Pc_i$ ($2 \le i \le n$) with $a,b \in \Pc_{i-1}$ (see notation in \ref{aspn:vine-to-Vposet}).
The \tbf{conditioning set} $D_{v_i}$ associated  with $v_i$ is defined by 
$$D_{v_i} := U_a \cap U_b,$$
and the \tbf{conditioned set} $C_{v_i}$ associated  with $v_i$ is defined by 
$$C_{v_i} := U_a \,\triangle\, U_b,$$
where $\triangle$ denotes the symmetric difference.
\end{definition}

It is easily seen that the nodes of an LR-vine satisfy the proximity condition. 
 The following properties were proved for an R-vine in \cite{BC02,KC03, KC06}. 
The arguments therein apply to a vine satisfying proximity condition as well since we only need the proximity of principal ideals.
 
\begin{lemma}
\label{lem:UCD}
Let $\Pc$ be a vine  of  rank $n$ and $v_i \in \Pc_i$ ($2 \le i \le n$). 
Suppose that the proximity condition holds. The following hold:
\begin{enumerate}[(a)]
\item  $|U_{v_i} (k) |= k + 1$ for $0 \le k \le i-1$. In particular,  $|U_{v_i}   |= i = \rk(v_i)$.
\item $|D_{v_i}   |= i- 2$ and $|C_{v_i}|=2$.
\end{enumerate}
\end{lemma}

We show below that local regularity and proximity of a vine are actually equivalent. 
 
\begin{proposition}
\label{prop:LR=P}
A vine $\Pc$  is  locally regular  if and only if the proximity condition holds for the nodes of $\Pc$. 
\end{proposition}

    \begin{proof}
It remains to show proximity implies local regularity. 
Let  $v \in \Pc$. 
We need to show that the principal ideal $\Pc_{\le v}$ itself is an R-vine.
Write $\Pc_{\le v}=(T_1,T_2,\ldots,T_{p})$ where $p= \rk(v) \le n$. 
By Lemma \ref{lem:UCD}(a), the rank $p$ and dimension $|U_{v}|$ of $\Pc_{\le v}$ are equal. 
Also by Lemma \ref{lem:UCD}(a), each forest  $T_{p-k}$ ($0\le k \le p-1$) has $k+1$ nodes and $k$ edges. 
Thus these forests must be trees. 
Clearly, proximity of a vine is preserved under taking ideals. 
It follows that  $\Pc_{\le v}$ is an R-vine.
\end{proof}

 \begin{lemma}
\label{lem:ind-subvine}
Let $\Pc=(F_1, \ldots,F_{n})$ be an LR-vine and $\V$ be the graphical vine defined by $\Pc$. 
Then  for every $a \in \Pc$, the ideal $\Pc_{\le a}$ coincides with the node poset of the induced subvine $\V[U_a]$. 
\end{lemma}
    \begin{proof}
Write $\V[U_a]=(F'_1, \ldots,F'_{q})$ and $\Pc_{\le a}=(T_1,\ldots,T_{p})$ where $T_i$'s  all are trees.
Note that $F'_1 = F_1[U_a]$ is a forest with at most $|U_a|-1$ edges. 
However, the tree $T_1$ is a subgraph of $F_1$ with node set $U_a$ that has exactly $|U_a|-1$ edges. 
Hence $F'_1 =T_1$. 
A repeated application of this argument yields $p=q$ and $F'_k = T_k$ for all $1\le k \le p$. 
Therefore, $\Pc_{\le a}$ is the node poset of  $\V[U_a]$.
 \end{proof}

  	\begin{corollary}
\label{cor:order-preserving}
 Let $\Pc$ be an LR-vine  and $a,b$ be nodes in $\Pc$.
If $U_a \subseteq U_b$ then $a \le b$.	
In particular,   if  $U_a = U_b$, then $a = b$.	
\end{corollary}

		\begin{proof}
Let $\V$ be the graphical vine defined by $\Pc$. 
If $U_a \subseteq U_b$, then $\V[U_a]$ is a subvine of $\V[U_b]$. 
By Lemma \ref{lem:ind-subvine}, $\Pc_{\le a}\subseteq\Pc_{\le b}$. 
Hence $a \le b$. 
If $U_a = U_b$, then by the first assertion, $a \le b$ and $b \le a$. 
Thus $a = b$.
	\end{proof}

\begin{remark} 
\label{rem:sub-LRvine}
By definition, an R-vine has a unique maximal element.  
Thus an ideal $\I$ of an LR-vine $\Pc$ is regular if and only if $\I$ is a principal ideal.
\end{remark}
 
If two posets $\Pc$ and $ \Pc'$ are isomorphic and $\Pc$ is an (L)R-vine, then $\Pc'$ is also an (L)R-vine. 
The result below enables us to represent a node in an LR-vine by its complete union (see  \ref{ex:D-vine} for an example). 
  		\begin{proposition}
	\label{prop:Phat} 
	Let $\Pc$ be an LR-vine. 
	Let $ \widehat\Pc$ be the poset consisting of the complete unions of the nodes in $\Pc$, i.e.~ $ \widehat\Pc =\{ U_a \mid a \in \Pc\}$ with partial order given by set inclusion.
		Define a map
	$$\eta_\Pc\colon \Pc  \longrightarrow \widehat\Pc  \quad \text{via} \quad a \mapsto U_a.$$
Then $\eta_\Pc$ is a poset isomorphism hence $\Pc \simeq  \widehat\Pc $. 
	\end{proposition}

	\begin{proof} 
	Clearly, $\eta_\Pc$ is a surjective homomorphism. 
	By Corollary \ref{cor:order-preserving}, for any $a,b \in \Pc$, if  $U_a= U_b$ then $a = b$.
	Thus $\eta_\Pc$ is injective hence  bijective.
	Again by Corollary \ref{cor:order-preserving}, for any $a,b \in \Pc$, $U_a \subseteq U_b$ if and only if $a \le b$.
	Thus the inverse of $\eta_\Pc$ is a poset homomorphism.
	We conclude that   $\eta_\Pc$ is an isomorphism.
	\end{proof}

Given a vine, it is important to know which induced subposet is again a vine. 
This motivates the following notion of \emph{truncation} of a vine \cite{BCA12}. 

\begin{definition}[Truncation] 
\label{def:truncation}
Let $(\Pc,\rk)$  be a finite graded poset of  rank $n$ and let $1\le k \le n$. 
The induced subposet $\overline\Pc_{\le k} := \{ x \in \Pc \mid   \rk(x) \le k \} = \bigcup_{i=1}^k \Pc_i$ with the rank function $\overline{\rk} = \rk$ is called the \tbf{$k$-lower truncation} of $\Pc$. 

Likewise, the induced subposet $\overline\Pc_{\ge k}:=   \{ x \in \Pc \mid   \rk(x) \ge k \} = \bigcup_{i=k}^{n} \Pc_i$ with the rank function $\overline{\rk}(v) = \rk(v) - k+1$ for all $v \in \overline\Pc_{\ge k}$ is called the \tbf{$k$-upper truncation} of $\Pc$. 

An induced subposet $\Qc$ of $\Pc$ is called a  \tbf{lower} (resp.~an  \tbf{upper})  \tbf{truncation} if $\Qc = (\overline\Pc_{\le k},\overline{\rk} )  $ (resp.~$\Qc = (\overline\Pc_{\ge k},\overline{\rk} )  $) for some $k$. 
A truncation $\Qc$ of $\Pc$ is called   \tbf{proper} if $\Qc \ne \Pc$.
\end{definition}

\begin{remark} 
\label{rem:low-upp}
Any lower truncation of a vine is an ideal. 
Hence by Remark \ref{rem:ideal-LRvine}, any lower truncation  of a vine (resp.~an LR-vine) is itself a vine (resp.~an LR-vine). 
However, a proper lower truncation  of an R-vine of rank $>1$ is not an R-vine. 
See Figure~\ref{fig:LRV} for an example of a lower truncation.

A proper upper truncation of a vine of rank $>1$ is not an ideal. 
However, proximity is preserved under taking  either upper or lower truncation.
Hence by Proposition \ref{prop:LR=P}, any upper truncation  of an LR-vine (resp.~a vine) is an LR-vine  (resp.~a vine). 
Unlike the lower truncation case, any upper truncation  of an R-vine is an R-vine by Remark \ref{rem:card-R-vine}. 
\end{remark}

The discussion above indicates that LR-vines are closed under either upper or lower truncation, while R-vines are only closed under upper truncation. 
We will see in \S\ref{subsubsec:SO} that these classes are also closed under  ``vertical" truncation, or more precisely, \emph{marginalization}.

   %********************************************************************************************************
\subsection{Construct an LR-vine from a given MAT-labeled graph}
\label{subsec:G-to-P}

	\begin{definition} 
\label{def:G-to-P}
Let $(G,\lambda)$ be an MAT-labeled graph with $N_G  =[\ell] $ and clique number $\omega(G)$.
Define a finite graded poset $\Pc=(\Pc, \le_\Pc, \rk_\Pc)$ from $(G,\lambda)$ as follows: 
\begin{enumerate}[(1)]
		\item $\Pc$ consists of the sets $\{i\}$ for $1 \le i \le \ell$ and all the principal cliques in $(G,\lambda)$ (Lemma \ref{lem:edge-clique}).
 	\item  For   $u,v \in \Pc$, $u \le_\Pc v$ if $u$ is a subset of $v$. 
		\item   $\rk_\Pc(v) = |v|$ for all $v \in \Pc$. 
	\end{enumerate}
\end{definition}

\begin{remark} 
\label{rem:constructed-LRvine}
It is easily seen that $\min(\Pc) =\{ \{i\} \mid 1 \le i \le \ell \}$.
The poset $\Pc$ is graded by $\rk_\Pc$ because by Lemma \ref{lem:pc-cover}, for every edge $e=\{i,j\}$ in $G$, the principal clique $K_e$ generated by $e$ covers exactly two principal cliques $K_e\setminus\{i\}$ and $K_e\setminus\{j\}$.
Note also that $\rk(\Pc)=\omega(G)$ by Lemma \ref{lem:largest-label}. 
\end{remark}

	\begin{theorem} 
	\label{thm:MAT-to-poset}
The poset $\Pc=(\Pc, \le_\Pc, \rk_\Pc)$  from Definition \ref{def:G-to-P}  is an LR-vine. 
In particular, if $(G,\lambda)$ is an MAT-labeled complete graph, then $\Pc$ is an R-vine.

	\end{theorem}
	
		\begin{proof}
		First we prove the first assertion. 
We argue by induction on the number $\ell$ of vertices of $G$. 
The assertion is clearly true when $\ell=1$. 
Suppose $\ell \ge 2$.
By Theorem \ref{thm:MAT-PEO}, there exists an MAT-PEO $(a_1,\ldots, a_\ell)$ of $(G,\lambda)$. 
Denote $G' := G \setminus a_\ell$ and $\lambda' :=  \lambda|_{E_{G'}} $. 
By Lemma \ref{lem:MAT-simplicial}, $(G',\lambda')$ is an MAT-labeled graph. 
Let $\Pc'$ be the  poset defined by $(G',\lambda')$ using Definition \ref{def:G-to-P}. 
By the induction hypothesis, $\Pc'$ is an LR-vine. 

Denote $d:= \deg_G(a_\ell)$. 
If $d=0$ i.e.~$a_\ell$ is an isolated vertex, then $\Pc= \Pc' \cup \{\{a_\ell\}\}$ is clearly an LR-vine. 
Suppose $d \ge 1$.
Write $\nbd_G(a_\ell) :=\{b_q \mid 1 \le q \le d\} \subseteq N_{G'}$ so that $\{a_\ell, b_q\} \in \pi_q$. 
For $1 \le q \le d$, define the following subgraph $H_q$ of $G$ 
$$H_q := G \setminus \{	\{a_\ell, b_{q+1}\},	\ldots,\{a_\ell, b_{d}\}	\}.$$
It is not hard to check that $a_\ell$ is an MAT-simplicial vertex in $(H_q, \lambda|_{E_{H_q}})$ for each $1 \le q \le d$. 
By Lemma \ref{lem:MAT-simplicial}, each $(H_q, \lambda|_{E_{H_q}})$ is an MAT-labeled graph since $G' = H_q \setminus a_\ell$. 

For each $1 \le q \le d$, let $\scR_q$ be the poset defined by $(H_q, \lambda|_{E_{H_q}})$ from Definition \ref{def:G-to-P}. 
We may write
$$\scR_q = \Pc' \cup \{v_0,\ldots,v_q\} = \Pc \setminus \{ v_{q+1}, \ldots, v_d\},$$
where $v_0 := \{a_\ell\}$ and each $v_q$ is the principal clique generated by $\{a_\ell, b_q\}$. 
We claim that  for each $1 \le q \le d$, the poset $\scR_q$ is an LR-vine.
In particular, the case $q=d$ whence $H_d=G$ and $\scR_d =\Pc$ yields the first assertion of Theorem \ref{thm:MAT-to-poset}. 

We argue by induction on $q$. 
The case $q=1$ is simple. 
The new non-minimal node $v_1 = \{a_\ell, b_1\}$ covers exactly two nodes $v_0=\{a_\ell\}$ and $\{b_1\}$. 
Also,  the new minimal node $v_0$ is only covered by $v_1$  since $a_\ell \notin v$ for all $v \in \Pc'$. 
The first associated forest of $\scR_1$ is given by that of $\Pc'$ with $a_\ell$ added so that $a_\ell$ is only connected to $b_1$. 
The second associated forest of $\scR_1$ is given by that of $\Pc'$ with an isolated vertex  $v_1$ added. 
The remaining associated forests of $\scR_1$ are the same as those of $\Pc'$. 
The proximity condition holds trivially. 
Hence $\scR_{1}$ is an LR-vine by Proposition \ref{prop:LR=P}. 

Suppose that the claim is true for some $1 \le q <d$. 
First note that $H_{q+1}\setminus\{a_\ell, b_{q+1}\}  = H_q$ and $\scR_{q+1}\setminus\{v_{q+1}\}  = \scR_q$.
Since $a_\ell$ is an MAT-simplicial vertex in $(H_q, \lambda|_{E_{H_q}})$, by (MS\ref{definition MAT-simplicial 3}), the edges in the complete subgraph $H_q[v_q]$ of $H_q$ induced by the vertices in $v_q$ have label $\le q-1$ except $\{a_\ell, b_q\} \in \pi_q$. 
Similarly, any edge of the form $\{b_{q+1},h\}$ where $h$ is a vertex in $v_q\setminus\{a_\ell\}$ has label $\le q$. 
Therefore, $v_q\setminus\{a_\ell\}$ consists of the $q$ conditioning vertices of $\{a_\ell, b_{q+1}\}$ in  $H_{q+1}$. 
Hence $v_{q+1} = v_q \cup \{b_{q+1}\}$. 

By Lemma \ref{lem:pc-cover}, $u:=v_{q+1} \setminus\{a_\ell\}$ is a principal clique (of cardinality $q+1$) in $(G',\lambda')$ hence a vertex in $\Pc'$.
Let $c$ be the vertex in $u$ such that $c$ is largest with respect to the MAT-PEO $(a_1,\ldots, a_{\ell-1})$ of $(G',\lambda')$. 
Since the edges in the complete subgraph $G'[u]$ have label $\le q$, by (MS\ref{definition MAT-simplicial 2}), there exists among these edges an edge $e $ incident on $c$ of label $q$. 
By the preceding paragraph, this edge $e$ must be $\{b_{q+1},h^*\}$ for some vertex $h^*$ in $v_q\setminus\{a_\ell\}$.
Moreover, such an edge $e$ is unique which is guaranteed by (ML\ref{definition MAT-labeling cycle}). 
Therefore, $u$ is generated by $e$. 
In particular, $u$ covers $v_q\setminus\{a_\ell\}$ in $\Pc'$ by Lemma \ref{lem:pc-cover}.
(See Figure \ref{fig:illustration} for a pictorial illustration of the proof.)

Now $v_{q+1}$ covers exactly two nodes $v_q$ and $u$ in $\scR_{q+1}$. 
Also, $v_q$ is covered  only by $v_{q+1}$  since $a_\ell \notin v$ for all $v \in \Pc'$.
The $(q+1)$-th associated forest of $\scR_{q+1}$ is given by that of $\scR_{q}$ with $v_q$ added so that $v_q$ is only connected to $u$. 
The $(q+2)$-th associated forest of $\scR_{q+1}$ is given by that of $\scR_{q}$  with an isolated vertex  $v_{q+1}$ added. 
The remaining associated forests of $\scR_{q+1}$ are the same as those of $\scR_{q}$.
This follows that $\scR_{q+1}$ is a vine. 
Furthermore, both $v_q \in \scR_{q}$ and $u \in \Pc'$ cover $v_q\setminus\{a_\ell\}$. 
Therefore, the proximity condition holds in $\scR_{q+1}$. 
Hence $\scR_{q+1}$ is an LR-vine by Proposition \ref{prop:LR=P}. 

Finally, we show the second assertion of Theorem \ref{thm:MAT-to-poset}. 
If $(G,\lambda)$ is an MAT-labeled complete graph, then $\rk(\Pc) =\dim(\Pc)=\ell$.
The proofs for the proximity of $\Pc$ and the fact the associated forests of $\Pc$ are trees run essentially along the same line as the proof of the first assertion. 
 	\end{proof}

	\begin{figure}[htbp!]
   \centering
\begin{tikzpicture}[scale=1]
\draw (0,-1) node[v](aell){} node [below]{$a_\ell$};
\draw (-2.2,2) node[v](bq+1){} node [above]{$b_{q+1}$};
\draw (1.2,1.8) node[v](bq){} node [right]{$b_{q}$};
\draw (0,2.3) node[v](h){} node [right]{$h^*$};

 \draw (0,3.5) node [above]{{\tiny $u$}};
  \draw (0,2.7) node [above]{{\tiny $v_q\setminus\{a_\ell\}$}};
  \draw (3,3.5) node [above]{$G'$};
  
  \draw (0,2) ellipse (4cm and 2cm);
    \draw[dashed] (0,2) ellipse (3cm and 1.5cm);
        \draw[dashed] (.3,2) ellipse (2cm and .7cm);
        
        \draw[blue, thick] (bq+1)--(aell) node [midway, left, blue] {{\tiny $q+1$}};
                \draw[red, thick] (bq)--(aell) node [midway, left, red] {{\tiny $q$}};
                       \draw[red, thick] (bq+1)--(h) node [midway, below, red] {{\tiny $q$}};

 \end{tikzpicture}
 \caption{\small A pictorial illustration of the proof of Theorem \ref{thm:MAT-to-poset}.}
\label{fig:illustration}
\end{figure}
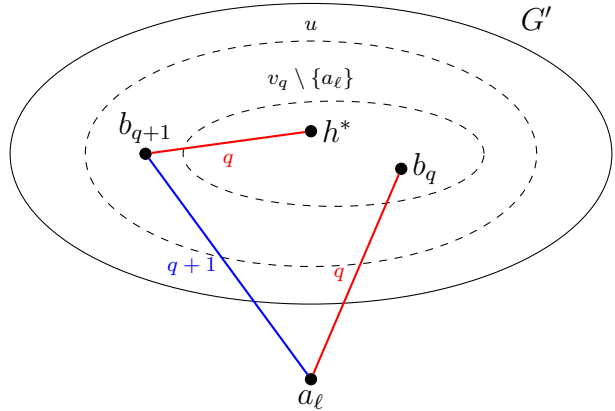

	\begin{corollary}
\label{cor:cond-sets}
 Given an MAT-labeled graph $(G,\lambda) $, let $\Pc$ denote the LR-vine defined by  $(G,\lambda)$ from Definition \ref{def:G-to-P}. 
 Then a  node $v \in \Pc$ has complete union $U_v = \{ \{a\} \mid a \in v\}$. 
 Moreover, if $v=K_e \in \Pc$ is a non-minimal node where $e=\{i,j\} \in E_G$, then $v$ has conditioned set $C_{v}=\{\{i\},\{j\}\}$.
\end{corollary}

		\begin{proof}
Use Remark \ref{rem:constructed-LRvine} and argue by an induction on $\rk(v)$. 
	\end{proof}
	
	We close this section by giving an example to illustrate the construction in Definition \ref{def:G-to-P} and Theorem \ref{thm:MAT-to-poset}. 
	First we need a definition.

 	\begin{definition}[D-vine] 
\label{def:D-vine}
		An R-vine is called a  \tbf{D-Vine} if each associated tree has the smallest possible number of vertices of degree $1$.
		Equivalently, each associated tree is a \emph{path graph}. 
 \end{definition}

\begin{remark} 
\label{rem:D-vine-RP}
 Let $\Phi$ be an irreducible root system in $\R^\ell$ with a fixed positive system $\Phi^+ \subseteq \Phi$ and the associated set of simple roots $\Delta= \{\alpha_1,\ldots,\alpha_\ell \}$. 
 Suppose that $\Phi$ is of type $A_\ell$ and the Hasse diagram of $\Phi$ is the path graph $\alpha_1\,\textendash\,\alpha_2\,\textendash\,\cdots\,\textendash\,\alpha_\ell$.
Then the positive roots of $\Phi$ are given by 
$$\Phi^+ = \left\{ \sum_{i\le k \le j } \alpha_k \,\middle\vert\,  1\le i\le j \le m\right\}.$$
 
It is not hard to show that the D-vine $\Pc$ with the first associated tree $ 1\,\textendash\,2\,\textendash\,\cdots\,\textendash\,\ell$ is isomorphic to the root poset $\scR (A_\ell)$ of type $A_\ell$ under the following isomorphism:
$$\mu \colon \Pc  \longrightarrow \scR (A_\ell)  \quad \text{via} \quad v \mapsto  \sum_{k \in U_v} \alpha_k .$$

\end{remark}

		\begin{example} 
			\label{ex:D-vine} 	
Figure~\ref{fig:D-vine} depicts a $4$-dimensional D-vine (middle) that can be constructed in three ways. 
First, it is the node poset of a graphical vine on $[4]$ (left) under the isomorphism in Proposition \ref{prop:Phat}.
Second, it is the poset defined an MAT-labeled complete graph (right) using Definition \ref{def:G-to-P}. 
Third, it is the root poset of type $A_4$ by Remark \ref{rem:D-vine-RP}.
The elements in the poset are written without set symbol for simplicity. 
The conditioned set of a non-minimal node is given to the left of the ``$\mid$" sign, while the conditioning set appears on the right.
For example, the top node $\{ 1,2,3,4\} $ (or the largest clique generated by $\{v_1,v_4\}$) is written by \textcolor{blue}{$14|23$}.
	\end{example}

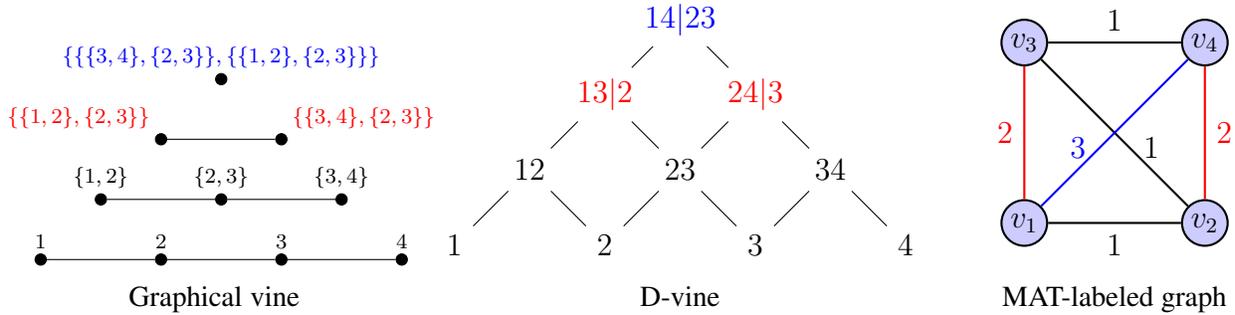
\begin{figure}[htbp!]
\begin{subfigure}{.35\textwidth}
   \centering
\begin{tikzpicture}[scale=.8]
\draw (0,3) node[v](14){} node [above]{\tiny\textcolor{blue}{$\{\{ \{3,4\}, \{2,3\}\}, \{ \{1,2\}, \{2,3\}\} \}$}};
  \draw (1,2) node[v](24){} node [above right]{\tiny\textcolor{red}{$\{ \{3,4\}, \{2,3\}\}$}};
  \draw (-1,2) node[v](13){} node [above left ]{\tiny\textcolor{red}{$\{ \{1,2\}, \{2,3\}\}$}};
  \draw (-2,1) node[v](12){} node [above]{\tiny $\{1,2\}$};
    \draw (0,1) node[v](23){} node [above]{\tiny $\{2,3\}$};
    \draw (2,1) node[v](34){} node [above]{\tiny $\{3,4\}$};    
    
      \draw (-1,0) node[v](2){} node [above]{\tiny$2$};
      \draw (1,0) node[v](3){} node [above]{\tiny$3$};
      \draw (-3,0)  node[v](1){} node [above] {\tiny$1$};
      \draw  (3, 0) node[v](4){} node [above] {\tiny$4$};   
  \draw (13) -- (24) ;
       \draw (12)-- (23) -- (34) ;
    \draw (1)--(2)--(3)--(4) ;
\end{tikzpicture}
\caption*{Graphical vine}
\end{subfigure}%
\begin{subfigure}{.4\textwidth}
   \centering
\begin{tikzpicture}[scale=1]
 
  \node (14) at (-1,3) {\textcolor{blue}{$14|23$}};
  \node (24) at (0,2) {\textcolor{red}{$24|3$}};
  \node (13) at (-2,2) {\textcolor{red}{$13|2$}};
  \node (34) at (1,1) {$34$};
  \node (23) at (-1,1) {$23$};
   \node (12) at (-3,1) {$12$};
         \node (4) at (2,0) {$4$};
       \node (3) at (0,0) {$3$};
     \node (2) at (-2,0) {$2$};
   \node (1) at (-4,0) {$1$};
   
  \draw (1)--(12) -- (13) -- (14) ;
    \draw (2)--(23) -- (24) ;
    \draw  (3)--(34) ;
  \draw (4)--(34) -- (24) -- (14) ;
      \draw (3)--(23) -- (13) ;
        \draw (2)--(12) ;
\end{tikzpicture}
\caption*{D-vine}
\end{subfigure}%
\begin{subfigure}{.3\textwidth}
  \centering
\begin{tikzpicture}[scale=1.2]
\begin{scope}[every node/.style={circle,thick,draw, inner sep=1.7pt, fill=blue!20}]
    \node (v1) at (0,0) {$v_1$};
    \node (v2) at (2,0) {$v_2$};
    \node (v4) at (2,2) {$v_4$};
    \node (v3) at (0,2) {$v_3$};
\end{scope}
\draw[thick] (v3)--(v2) node [near end, above] {$1$};
\draw[red, thick] (v4)--(v2) node [midway, right, red] {$2$};
\draw[thick] (v4)--(v3) node [midway, above] {$1$};
\draw[blue, thick] (v4)--(v1) node [near end, above, blue] {$3$};
\draw[thick] (v2)--(v1) node [midway, below] {$1$};
\draw[red, thick] (v3)--(v1) node [midway, left, red] {$2$};
\end{tikzpicture}
  \caption*{MAT-labeled graph}
\end{subfigure}
\caption{\small An MAT-labeled complete graph on $4$ vertices (right), the  D-vine  (middle) ($=$ type $A$ root poset) defined by the graph using Definition \ref{def:G-to-P}, and the corresponding  graphical vine (left).}
\label{fig:D-vine}
\end{figure}

\begin{remark} 
\label{rem:Pascal}
Max Wakefield let us know an interesting observation that the D-vine is isomorphic to the intersection lattice (with bottom element removed) of the \emph{Pascal arrangement} introduced in \cite{MW11}. 
We leave a possible generalization of the main result in \cite{MW11} to an R-vine for future research. 

\end{remark}

  %********************************************************************************************************

\section{From LR-vines to MAT-labeled graphs}
\label{sec:vine-to-MAT}
Constructing an MAT-labeled graph from a given LR-vine needs more effort.

   %********************************************************************************************************
\subsection{Some new properties of LR-vines}
\label{subsec:newprop}
The following lemma provides the key ingredient of our construction.
\begin{lemma}[Joining path] 
\label{lem:LR-paths}
Let $\Pc=(F_1,\ldots,F_{n})$ be an $\ell$-dimensional LR-vine. 
Let $i, j \in  \min(\Pc)=[\ell]$ be distinct minimal nodes. 
Let $v \in \Pc_{r}$ be a non-minimal node ($2 \le r \le n$). 
The following are equivalent: 
\begin{enumerate}[(1)]
		\item  $v =i \vee j$, i.e.~$v$ is the join of $i$ and $j$.
		\item $C_v = \{i,j\}$, i.e.~$\{i,j\}$ is the conditioned set of $v$.
		\item There exist $r$ paths (uniquely determined by $i$ and $j$) $P_k=P_k(i,j) \in E_k$   in the forests $F_k$  ($1 \le k \le r$)  satisfying the following three conditions: 
\begin{enumerate}[(a)]
		\item $P_1=(a_{1,1}, a_{1,2}, \ldots, a_{1,p_1}), p_1 \ge 2$ is a (unique) path connecting nodes $a_{1,1}:=i$ and $a_{1,p_1}:=j$ in $F_1$.
		\item For $2 \le k \le r$, $P_k=(a_{k,1}, a_{k,2},\ldots, a_{k,p_k}), p_k \ge 1$ is a (unique) path connecting nodes $a_{k,1}:=\{ a_{k-1,1}, a_{k-1,2}\}$ and $a_{k,p_k}:=\{ a_{k-1,p_{k-1}-1}, a_{k-1,p_{k-1}}\}$ in $F_k$.
		\item $P_{r}=v$.
		
	\end{enumerate}
	\end{enumerate}
In this case, we call the path $P_k(i,j) \in E_k$ ($1 \le k \le r$) the \tbf{$k$-joining path} of the pair $\{i,j\}$ (or $P_k(v)$ the $k$-joining path of $v$).
\end{lemma}

Before giving the proof of Lemma \ref{lem:LR-paths}, we address some remarks. 
\begin{remark}
\label{rem:Rvine-jp}
 If $\Pc=(T_1,\ldots,T_{\ell})$ is an R-vine, then for any distinct $i, j \in [\ell]$, the joining paths of $ \{i,j\}$ always exist since $T_k$ is a tree for each $1 \le k \le \ell$ and $T_{\ell}$ has only one node. 
\end{remark}

The implication $(2) \Leftarrow (3)$ was stated in the proof of \cite[Lemma 3.9]{Diss10} for R-vines. 
Unfortunately, the proof was not complete. 
We give below a complete proof that works for an arbitrary LR-vine.

   \begin{proof}[Proof of Lemma \ref{lem:LR-paths}.]
   First we prove $(2) \Leftarrow (3)$. 
   By definition, $P_{r-1}=\{a_{r-1,1}, a_{r-1,p_{r-1}}\}=v$.
   We need to show $C_{v} = U_{a_{r-1,1}} \,\triangle\, U_{a_{r-1,p_{r-1}}} = \{i, j\}$. 
   Since $|C_{v}|=2$ by Lemma \ref{lem:UCD}(b), it is enough to show that $i \le a_{r-1,1} , i \not\le a_{r-1,p_{r-1}}$ and $j \le a_{r-1,p_{r-1}}, j \not\le a_{r-1,1}$.
 This follows once we prove that for each $1 \le k \le r $ the following two statements hold:
 \begin{enumerate}[(S1)]
		\item $i\le a_{k,1}$ and $i\not\le b$, where $b$ is any node in $F_k$ such that there exists a (unique) path in $F_k$ connecting $b$ and some non-initial node $a_{k,t}$ ($2 \le t \le p_k$) of $P_k$ but not passing through its initial node $a_{k,1}$. \label{S1} 
		
		\item $j\le a_{k,p_k}$ and $j \not\le  c$, where $c$ is any node in $F_k$ such that there exists a (unique)  path in $F_k$ connecting $c$ and some non-final node $a_{k,t}$ ($1 \le t \le p_k-1$) of $P_k$ but not passing through its final node $a_{k,p_k}$.\label{S2} 
	\end{enumerate}
 
 Since the roles of $i$ and $j$ are symmetric, it suffices to prove Statement (S\ref{S1}). 
 First we need the following crucial property of the paths $P_k$'s. 
 
    \begin{claim} 
\label{cla:pass-thru} 
 Fix $1 \le k \le r - 1$. 
Suppose $P_k$ is given by $P_k =(\alpha_{1}, \alpha_{2},\ldots, \alpha_{p})$ for $p \ge 1$ in terms of its node sequence. 
Then node set of $P_{k+1}$ consists of all the edges $\{\alpha_{t}, \alpha_{t+1}\}$ for $1 \le t \le p-1$ of  $P_k$, and some edges of the form $\{\alpha_{s+1}, \mu\}$ for $1 \le s \le p-2$ and $\mu \in \nbd_{F_k}(\alpha_{s+1})\setminus \{\alpha_{s},\alpha_{s+2}\}$. 
\end{claim}

\begin{proof}[Proof of Claim \ref{cla:pass-thru}.]
  By definition, the initial and final nodes in $P_{k+1}$ are $\{\alpha_{1}, \alpha_{2}\}$ and $\{\alpha_{p-1}, \alpha_{p}\}$, respectively. 
  The claim is clear true for $p\le3$. 
Suppose $p \ge 4$. 
By the proximity condition, the length of $P_{k+1}$ is at least two. 
Again by the proximity, the node adjacent to the initial node in $P_{k+1}$ must have the form either $\{\alpha_{1}, \beta\}$ where $\beta \in \nbd_{F_k}(\alpha_{1}) \setminus \{\alpha_2\}$, or $\{\alpha_{2}, \gamma\}$ where $\gamma \in \nbd_{F_k}(\alpha_{2})\setminus \{\alpha_1\}$. 

 The former cannot occur, otherwise arguing on the proximity yields two different paths in $F_k$ connecting $\alpha_{1}$ and  $\alpha_{p}$; one is $P_k$ passing through $\alpha_{2}$, the other passing through some $\beta' \in \nbd_{F_k}(\alpha_{1}) \setminus \{\alpha_2\}$. 
 Thus the latter occurs, and $P_{k+1}$ possibly continues to pass through some node of the same form as $\{\alpha_{2}, \gamma\}$. 
 The following two conditions hold: The path $P_{k+1}$ must (i) reach the node $\{\alpha_{2}, \alpha_{3}\}$ and after that (ii) does not pass through any further node of this form. 
 If either (i) or (ii) fails, then there are two different paths in $F_k$ connecting $\alpha_{2}$ and  $\alpha_{p}$. 
 Hence $P_{k+1}$ passes through some node of the form $\{\alpha_{3}, \delta\}$ where $\delta \in \nbd_{F_k}(\alpha_{3})\setminus \{\alpha_2\}$ until it reaches $\{\alpha_{3}, \alpha_{4}\}$ and so on.

  A repeated application of the argument above completes the proof of the claim.
  \begin{align*}
  P_{k}: &\, \alpha_{1}  \rightarrow \alpha_{2}  \rightarrow  \alpha_{3} \rightarrow \cdots  \rightarrow \alpha_{p},\\
P_{k+1}: &\, \{\alpha_{1}, \alpha_{2}\} \rightarrow  \{\alpha_{2}, \gamma\}  \rightarrow \cdots  \rightarrow \{\alpha_{2}, \alpha_{3}\}  \rightarrow \cdots  \rightarrow \{\alpha_{3}, \alpha_{4}\} \rightarrow \cdots  \rightarrow \{\alpha_{p-1}, \alpha_{p}\}.\qedhere
   \end{align*}
\end{proof}

Now we return to the proof of (S\ref{S1}).   
The first part is easy since by definition, $a_{1,1}=i$ and $a_{k,1} \le a_{k+1,1}$ for all $1 \le k \le r - 1$.
 We argue the second part by induction on $k$. 
 The statement is clear true for $k=1$. 
 Suppose it is true for any $1 \le k <r$. 
 Let $b=\{b_1,b_2\}$ be an arbitrary node in $F_{k+1}$ such that there exists a path in $F_{k+1}$ connecting $b$ and some non-initial node  of $P_{k+1}$ but not passing through its initial node. 
 By the relation of the paths $P_{k}$ and $P_{k+1}$ proved in Claim \ref{cla:pass-thru}, there exists a path in $F_{k}$ connecting $b_s$ ($s=1$ or $2$) and some non-initial node  of $P_{k}$ but not passing through its initial node. 
By the induction hypothesis, we must have $i \not\le b_1$ and $i\not\le b_2$. 
It follows that $i\not\le b$. 
This completes the proof of (S\ref{S1}) and hence the proof of $(2) \Leftarrow (3)$.  

To prove  $(2) \Rightarrow (3)$, the following fact is  useful. 

\begin{remark}
\label{rem:Rvine-only1}
Suppose $\Pc$ is an R-vine. 
By Remark \ref{rem:Rvine-jp}, for any $1 \le i\ne j \le \ell$, the joining paths  of $ \{i,j\}$ always exist. 
Thus by $(2) \Leftarrow (3)$, there exists a non-minimal node $v \in \Pc$ such that $C_v = \{i,j\}$. 
Moreover, by Remark \ref{rem:card-R-vine}, the number of non-minimal nodes in  $\Pc$ is equal to $\ell(\ell-1)/2$. 
Therefore, every pair of distinct elements in $[\ell]$ occurs exactly once as the conditioned set of a non-minimal node.
\end{remark}

Now we give the proof of $(2) \Rightarrow (3)$. 
Write $\Pc_{\le v}=(T_1,T_2,\ldots,T_{r})$. 
 By definition, $\Pc_{\le v}$ is an R-vine.
 If $C_v = \{i,j\}$, then $i$ and $j$ are nodes in $T_1$. 
 By Remark \ref{rem:Rvine-jp} and $(2) \Leftarrow (3)$, there exist $r'$ joining paths of $ \{i,j\}$ hence a node $v'$ in $\Pc_{\le v}$ such that $C_{v'} = \{i,j\} =C_v $. 
Hence by Remark \ref{rem:Rvine-only1}, $v=v'$ and $r=r'$. 

Next we show $(1) \Leftarrow (3)$. 
By Statements (S\ref{S1}) and  (S\ref{S2}), $i \le v$ and $j \le v$. 
Let $u\in \Pc_s$ for $1 \le s \le n$ be a node such that $i \le u$ and $j \le u$. 
In particular, $i$ and $j$ are minimal nodes in the R-vine $\Pc_{\le u}$. 
By Remark \ref{rem:Rvine-jp}, there exist $s$ joining paths of $ \{i,j\}$ in $\Pc_{\le u}$. 
By the uniqueness of the paths in the associated forests $F_k$'s, the joining paths of $ \{i,j\}$ in $\Pc_{\le u}$ and $\Pc_{\le v}$ must be the same. 
Thus $r=s$ and $v \le u$. 
Hence $v$ is the join of $i$ and $j$.

Finally we show $(1) \Rightarrow (3)$. 
Suppose  $v =i \vee j$. 
Hence there exist $s'$ joining paths of $ \{i,j\}$ hence a node $u'$ in $\Pc_{\le v}$ such that $i \le u'$ and $j \le u'$. 
By the definition of a join, we must have $v=u'$ and $s'=r$.
\end{proof}

\begin{remark} 
\label{rem:missing} 
The missing piece of the proof of \cite[Lemma 3.9]{Diss10} is that, the fact that $i \not\le a$ for any non-initial node $a$ in $P_k$ does not automatically imply that $i \not\le b$ for any non-initial node $b$ in $P_{k+1}$. 
\end{remark} 
 
		\begin{example} 
			\label{ex:jp} 	
Let $\Pc$ be the $4$-dimensional D-vine in Figure~\ref{fig:D-vine}. 
Let $i=1$ and $j=4$ be minimal nodes. 
The joining paths of $1$ and $4$ are given by $P_1 = (1,2,3,4)$, $P_2 = (12,23,34)$, $P_3 =(\textcolor{red}{13|2}, \textcolor{red}{24|3})$, $P_4 =\textcolor{blue}{14|23} $. 
The join of $1$ and $4$ is $\textcolor{blue}{14|23} $. 
The conditioned set of $\textcolor{blue}{14|23} $ is $\{1,4\}$. 
These calculations are consistent with Lemma \ref{lem:LR-paths}.
	\end{example}

 The first two corollaries below are immediate consequences of Lemma \ref{lem:LR-paths} and Claim \ref{cla:pass-thru}, respectively.

\begin{corollary}
\label{cor:C=C}
Let  $u,v$ be two non-minimal nodes in  an   LR-vine $\Pc$. 
  If  $C_u = C_v$, then $u = v$.	
\end{corollary}

\begin{corollary}
\label{cor:lessthanv}
Let $v$ be a non-minimal node in an   LR-vine $\Pc$. 
Then all the nodes of the joining paths of $v$ are in $\Pc_{\le v}$. 
\end{corollary}

\begin{corollary}
\label{cor:conditioning}
	Let $\Pc $ be an   LR-vine. 
	Let $v$ be a non-minimal node in $\Pc$ with conditioned set $C_v =\{i,j\}$ for $i, j \in [\ell]$. 
	If $k$ is in the conditioning set of $v$, then  the joins $u= i \vee k$ and $w=j  \vee k$  exist in $\Pc$.
	Moreover, $v>u$ and $v>w$.
	
	\end{corollary}
	
		\begin{proof}
If $k \in D_v$, then $k$ is a minimal node of the R-vine $\Pc_{\le v}$. 
 By Remark \ref{rem:Rvine-jp},  the joining paths  of $ \{i,k\}$ and $ \{k,j\}$ always exist. 
 Thus by Lemma \ref{lem:LR-paths}, the joins $u= i \vee k$ and $w=j  \vee k$ exist in $\Pc_{\le v}$ hence in $\Pc$. 
Note that $v \ne u$ and $v \ne w$ since $C_v \ne C_u$ and $C_v \ne C_w$. 
Thus  $v>u$ and $v>w$.

	\end{proof}

 We give some further properties of the joining paths of a non-minimal node in an LR-vine.

	\begin{lemma}
	\label{lem:subpath} 
	Let $\Pc =(F_1,\ldots,F_{n})$ be an   LR-vine  and let $u,v$ be two distinct non-minimal nodes in $\Pc$. 
Suppose that there exists a number $k$ ($1\le k\le \rk(u)$) such that the $k$-joining path $P_k(u)$ is a proper subpath of the $k$-joining path $P_k(v)$. 
Then $P_q(u)$ is a proper subpath of $P_q(v)$ for all $k\le q\le \rk(u)$. 
In particular, $v>u$. 
	\end{lemma}
	\begin{proof}
	Write $P_k(v)=(a_{k,1}, a_{k,2},\ldots, a_{k,p_k}), p_k \ge 2$. 
	Since $P_k(u)$ is a proper subpath of $P_k(v)$, we may write  $P_k(u)=(a_{k,s},\ldots, a_{k,t})$, where $1\le s\le t\le p_k$ and $(s,t)\neq (1,p_k)$. 
	Note that $P_{k+1}(u)$ is the unique path connecting $\{a_{k,s},a_{k,s+1}\}$ and  $\{a_{k,t-1},a_{k,t}\}$ in the forest $F_{k+1}$. 
	Moreover, by Claim \ref{cla:pass-thru}, $P_{k+1}(v)$ passes through $\{a_{k,s},a_{k,s+1}\}$ and  $\{a_{k,t-1},a_{k,t}\}$.
	Therefore, $P_{k+1}(u)$ must be a proper subpath of $P_{k+1}(v)$. 
	Apply this argument repeatedly, we may show that $P_q(u)$ is a proper subpath of $P_q(v)$ for all $k\le q\le \rk(u)$. 
In particular, the case $q=\rk(u)$ yields $v>u$. 
	\end{proof}
	
	Before giving the next property in  Lemma \ref{lem:m-cycle-LR}, we need a technical lemma on paths in a forest. 
	
	\begin{lemma} 
	\label{lem:m-cycle} 
	Let $F$ be a forest.  
	Let $i_1,i_2, \ldots, i_m$ for $m \ge 3$ be mutually distinct nodes  in $F$. 
	For each $1 \le s \le m$, suppose that there exists a (unique) path $P_{s,s+1}$ in $F$  connecting $i_s$ and $i_{s+1}$. 
		Here we take the index modulo $m$. 
			Denote by $e'_{s,s+1}$ and $e''_{s,s+1}$ the edges in $P_{s,s+1}$ incident on $i_s$ and $i_{s+1}$, respectively. 
			Suppose that there exists $1 \le t \le m$ such that $e''_{t,t+1} \ne e'_{t+1,t+2}$. Then among the paths $P_{s,s+1}$'s for $s \notin \{t,t+1\}$ there exist two paths $P_{a,a+1}$ and $P_{b,b+1}$ (not necessarily distinct) both of length $\ge2$ containing   $e''_{t,t+1}$ and  $e'_{t+1,t+2}$, respectively. 
	\end{lemma}
	\begin{proof} 
	Let $T$ denote the subgraph of $F$ induced by the vertices of the paths $P_{s,s+1}$'s for all $1 \le s \le m$. 
	Note that $T$ is a connected subgraph of $F$ hence $T$ is a tree. 
	
	If $m=3$, then the path $P_{t,t+2}$ of length $\ge2$ contains both $e''_{t,t+1}$ and  $e'_{t+1,t+2}$. 
	Suppose $m>3$. 
	By the assumption, the concatenation of  $P_{t,t+1}$ and  $P_{t+1,t+2}$, denoted $P$, is the unique path in $T$ connecting $i_{t}$ and $i_{t+2}  $.
	Moreover, there exists a walk in $T$ connecting $i_{t}$ and $i_{t+2}  $ whose edge set is the union of the edge sets of $P_{t+2,t+3}$, $P_{t+3,t+4}, \ldots, P_{t-1,t} $. 
	By Lemma \ref{lem:walk-path}, this walk must contain the path $P$. 
	In particular, the edge $e''_{t,t+1}$ (resp.~$e'_{t+1,t+2}$) is contained in a path $P_{a,a+1}$ (resp.~$P_{b,b+1}$) for some $a, b\notin \{t,t+1\}$. 
 Clearly, both $P_{a,a+1}$ and $P_{b,b+1}$ have lengths $\ge2$.
\end{proof}

\begin{notation}
\label{nota:vxy} 
Let $\Pc$  be  an LR-vine.
	In what follows, for two distinct minimal nodes $i ,j \in \min(\Pc)$, if the join $i \vee j$  exists in $\Pc$, we denote $v_{i,j} := i \vee j \in \Pc$. 
Sometimes, two minimal nodes in $\Pc$ are  denoted by $i_s,i_t$, in which case we write  $v_{s,t} :=v_{i_s,i_t}$. 
		Note that by  Lemma \ref{lem:LR-paths}, the nodes $v_{i,j}$'s are mutually distinct, i.e.~if $\{i,j\} \ne \{i',j'\}$ then $v_{i,j} \ne v_{i',j'}$.
\end{notation} 

	\begin{lemma} 
	\label{lem:m-cycle-LR} 
	Let $\Pc =(F_1,\ldots,F_{n})$ be an   LR-vine of rank $n$. 
	Let $i_1,i_2, \ldots, i_m  \in  \min(\Pc) $ for $m \ge 3$ be mutually distinct minimal nodes  in $\Pc$. 
	Suppose that the join $v_{s,s+1} $  exists in $\Pc$ for each $1 \le s \le m$. 
				Here again we  take the index modulo $m$. 
	Then there exist a node $i_t$ and a join $v_{a,a+1} $ for $a,t \in [m]$ such that  $a \notin \{t-1,t\}$ and $i_t$ belongs to the conditioning set of $v_{a,a+1}$.

	\end{lemma}
	\begin{proof} 
 By  Lemma \ref{lem:LR-paths}, for each $1 \le s \le m$, there exist the $k$-joining paths $P_k(v_{s,s+1})$'s of  $v_{s,s+1} $ for $1 \le k \le \rk(v_{s,s+1})$. 
 In particular, $P_1(v_{s,s+1})$ is the unique path in $F_1$  connecting $i_s$ and $i_{s+1}$. 
			For each $1 \le s \le m$, denote by $e'_{s,s+1}$ and $e''_{s,s+1}$ the edges in $P_1(v_{s,s+1})$ incident on $i_s$ and $i_{s+1}$, respectively. 
   If there exists $t \in [m]$ such that $e''_{t,t+1} \ne e'_{t+1,t+2}$, then by Lemma \ref{lem:m-cycle}, $e''_{t,t+1}$  is an edge of a path $P_1(v_{a,a+1})$ of length $\ge2$  for some $a\in [m]\setminus \{t,t+1\}$. 
   By Claim \ref{cla:pass-thru}, $e''_{t,t+1}$  is a node in $P_2(v_{a,a+1})$. 
   By Corollary \ref{cor:lessthanv}, $i_{t+1} < e''_{t,t+1} \le v_{a,a+1}$.
   Hence $i_{t+1}$ belongs to the conditioning set of $v_{a,a+1}$.

   Now consider the case $e''_{s-1,s} = e'_{s,s+1}$  for each $1 \le s \le m$.
Denote this common edge by $j_{s}$.	
Note that these edges become nodes in $F_2$. 
By definition, $P_2(v_{s,s+1})$ is the path in $F_2$ connecting $j_{s}$ and $j_{s+1}$ for each $1 \le s \le m$. 
Denote by $f'_{s,s+1}$ and $f''_{s,s+1}$ the edges in $P_2(v_{s,s+1})$ incident on $j_{s}$ and $j_{s+1}$, respectively. 
By a similar argument as in the preceding paragraph with the aid of Lemma \ref{lem:m-cycle}, we only need to consider the case  $f''_{s-1,s} = f'_{s,s+1}$ for all $1 \le s \le m$.

A repeated application of this argument leads us to the situation that for every $k$, there exist mutually distinct nodes $h_1,h_2, \ldots, h_m$ in $F_k$ such that the $k$-joining path $P_k(v_{s,s+1})$ connects $h_{s}$ and $h_{s+1}$ for each $1 \le s \le m$. 
Furthermore, if $g'_{s,s+1}$ and $g''_{s,s+1}$ are the edges in $P_k(v_{s,s+1})$ incident on $h_{s}$ and $h_{s+1}$, respectively, then $g''_{s-1,s} = g'_{s,s+1}$ for each $1 \le s \le m$.

However, let $q \in [m]$ such that $\rk(v_{q,q+1}) =\min \{\rk(v_{s,s+1})  \mid  1 \le s \le m\}$ and let $k=\rk(v_{q,q+1})-1$. 
Then the path $P_k(v_{q,q+1})$ has length $1$ (or simply an edge in $F_k$).	
The situation in the paragraph above implies that $P_k(v_{q,q+1})$ is a proper subpath of $P_k(v_{q-1,q})$. 
By Lemma \ref{lem:subpath},  $i_{q+1} < P_k(v_{q,q+1}) \le v_{q-1,q}$.
   Hence $i_{q+1}$ belongs to the conditioning set of $ v_{q-1,q}$. 
\end{proof}

We have a stronger statement when $m=3$ in  Lemma \ref{lem:m-cycle-LR}. 
First we address a remark. 
 
	\begin{remark} 
			\label{rem:conca} 
				Let $F$ be a forest.  
	Let $i_1,i_2$, and $i_3$ be three mutually distinct nodes  in $F$. 
	Suppose there exist the paths $P_{1,2}, P_{2,3}, P_{3,1}$ in $F$ connecting $i_1$ and $i_2$, $i_2$ and $i_3$, $i_3$ and $i_1$, respectively.  
	Let $T$ denote the subgraph (tree) of $F$ induced by the vertices of the paths $P_{1,2}, P_{2,3}, P_{3,1}$.   
	One may show that there does not exist a path among $P_{{1,2}}$, $P_{2,3}$ and  $P_{3,1}$ being the concatenation of the other two paths if and only if $i_1,i_2, i_3$ are leaves in $T$.
	In particular, it is not the case if one of the paths has length $1$.	
	\end{remark}

	\begin{lemma} 
	\label{lem:conca} 
	Let $\Pc =(F_1,\ldots,F_{n})$ be an   LR-vine of rank $n$. 
	Let $a_1,b_1, c_1  \in  \min(\Pc) $ be mutually distinct minimal nodes  in $\Pc$. 
	Suppose that the joins $u=a_1 \vee b_1$, $v=b_1 \vee c_1$, and $w=a_1 \vee c_1$ exist in $\Pc$. 
	Then there exists a number $k$ $(1\le k\le n$) such that one of the three $k$-joining paths  $P_k(u)$, $P_k(v)$ and  $P_k(w)$ is the concatenation of the other two paths. 
	
	As a consequence, there exists a node among three nodes $u,v$ and $w$ strictly greater than the other two.

	\end{lemma}
	\begin{proof} 
	The proof of the first assertion is similar to the proof of Lemma \ref{lem:m-cycle-LR} with the use of Remark \ref{rem:conca} in place of  Lemma \ref{lem:m-cycle}. 
	The second assertion follows from the first and Lemma \ref{lem:subpath}.
\end{proof}

   %********************************************************************************************************
\subsection{Construct an MAT-labeled graph from a given LR-vine}

	\begin{definition} 
\label{def:P-to-G}
Let $\Pc$ be an   LR-vine of dimension $\ell$ and rank $n$. 
Define an edge-labeled graph  $(G,\lambda) $ from $\Pc$ as follows: 
\begin{enumerate}[(1)]
		\item The vertex set $N_{G}$ is given by the set of minimal nodes, i.e.~$N_G  :=\min(\Pc)=[\ell]$.
		\item The edge set $E_{G}$ is given by the conditioned sets of non-minimal nodes, i.e.
		\begin{align*}
E_G & :=\{C_v \mid v \in \Pc\setminus \min(\Pc)\} \\
& =\{ \{i,j\} \subseteq [\ell] \mid i\ne j \text{ and the join $v_{i,j} = i \vee j$  exists in $\Pc$} \}.
\end{align*}
(The second expression of $E_G$ above follows from Lemma \ref{lem:LR-paths}.)
		\item The labeling $\lambda: E_G  \longrightarrow \mathbb{Z}_{>0}$ is defined by 
	$$\lambda(i,j):= \rk_\Pc(v_{i,j} )-1.$$ 
	\end{enumerate}
  
\end{definition}

	\begin{theorem} 
	\label{thm:poset-to-MAT}
	The edge-labeled graph $(G,\lambda)$   from Definition \ref{def:P-to-G} is an MAT-labeled graph. 
	In particular, if $\Pc$ is an  R-vine, then $(G,\lambda)$ is an MAT-labeled complete graph. 

	\end{theorem}
		\begin{proof}
The first assertion follows from  Lemmas   \ref{lem:ML1}  and \ref{lem:ML2}  below. 
The second assertion follows from the first and Remark \ref{rem:Rvine-only1}.
	\end{proof}

	\begin{lemma}
		\label{lem:ML1} 
 $(G, \lambda)$ satisfies   (ML\ref{definition MAT-labeling cycle}). 
	\end{lemma}

	\begin{proof}
	Suppose to the contrary that there exist $1 \le k \le n-1$ and an  $m$-cycle $C_m$ for $m \ge 3$ with edges $\{i_1,i_2\}$, $\{i_2,i_3\},\ldots,\{i_m,i_1\}$ such that $ \lambda(i_1,i_2) < k$ and $ \lambda(i_{s},i_{s+1})=k$ for $2 \le s \le m$. 
		Here we take the index modulo $m$. 
We choose the smallest such $m$.
	
	If $m=3$, then by Lemma \ref{lem:conca}, there exists a node among the nodes $v_{1,2}$, $v_{2,3}$, $v_{3,1}$ strictly greater than the other two. 
	This is a contradiction since there are two nodes of the same rank $k+1$, while the remaining node has rank $< k+1$. 
We may assume $m>3$. 

By Lemma \ref{lem:m-cycle-LR}, there exist a minimal node $i_t$ and a join $v_{a,a+1}= i_a \vee i_{a+1}$ in $\Pc$ for $a,t \in [m]$ such that  $t \notin \{a,a+1\}$ and $i_t$ belongs to the conditioning set of $v_{a,a+1}$. 
 By Corollary \ref{cor:conditioning}, the joins $v_{a,t}$, $v_{t,a+1}$ exist in $\Pc$ and both are strictly smaller than $v_{a,a+1}$. 
 Hence the edges $\{i_a,i_t\}$, $\{i_t,i_{a+1}\}$ exist in  $(G, \lambda)$ and both have labels $<k$. 
 Therefore, there exists a cycle  in  $(G, \lambda)$ of length $<m$ with one edge of label $<k$ and the other edges of the same label $k$. 
This contradicts the minimality of $m$. 
Thus for every $ k \in \mathbb{Z}_{>0} $, an edge  $e \in \pi_{< k} $ does not form a cycle with edges in $\pi_k$.

Now suppose $(G, \lambda)$ contains an  $m$-cycle $C_m$ for $m \ge 3$ with all edges of the same label $k$ for some $1 \le k \le n-1$. 
By Lemma \ref{lem:conca}, we may assume $m>3$. 
By a similar argument as in the preceding paragraph with the aid of Lemma \ref{lem:m-cycle-LR}, the cycle $C_m$ has a chord of label $<k$. 
This contradicts the conclusion of the preceding paragraph.
	\end{proof}

	\begin{lemma}
	\label{lem:ML2} 
		Fix $1 \le k \le n-1$ and let $\{i,j\} \in \pi_k$. 
		Then  the conditioning set of $v_{i,j} =i \vee j$ coincides with the set of conditioning vertices of  $\{i,j\}$. 
		In particular, 
 $(G, \lambda)$ satisfies  (ML\ref{definition MAT-labeling triangle}). 
	\end{lemma}

	\begin{proof} 	
	Note that $v_{i,j} \in \Pc_{k+1}$. 
 By Lemma \ref{lem:UCD}(b), the conditioning set $D_{v_{i,j}}$ of $v_{i,j}$ contains $k-1$ elements. 
 We may write it as $D_{v_{i,j}} = \{h_t \mid  1\le t\le k-1\}$. 
 By Corollary \ref{cor:conditioning}, the joins $v_{i,h_t}$, $v_{h_t,j}$  for $1\le h_t\le k-1 $ exist in $\Pc$ and all are strictly smaller than  $v_{i,j}$. 
 In particular,  the edges $\{i,h_t\}$, $\{h_t,j\}$ have label $< k$ for all $h_t $. 
 This implies that $D_{v_{i,j}}$ is contained in the set of conditioning vertices of  $\{i,j\}$.

	Now let $h \in [\ell] \setminus U_{v_{i,j}}$, i.e.~$h \not\le v_{i,j} $. 
	We show that if the joins $v_{i,h}$ and $v_{h,j}$ both exist, then either $\{i, h\}$ or $\{h, j\}$ has label $> k$. 
	It cannot happen that $v_{i,h} < v_{i,j}$ or $v_{h,j} < v_{i,j}$; otherwise $h \le v_{i,j}$. 
	Hence by Lemma \ref{lem:conca}, 
	either $v_{i,h} > v_{i,j}$ or $v_{h,j} > v_{i,j}$. 

Thus the conditioning set of $v_{i,j}$ coincides with the set of conditioning vertices of  $\{i,j\} \in \pi_k$. 
Since this is true for every $k$, we conclude that  $(G, \lambda)$ satisfies  (ML\ref{definition MAT-labeling triangle}). 
	\end{proof}

 %********************************************************************************************************
\section{Equivalences of categories}
\label{sec:equi-cat}

For basic definitions and notations of category theory, we refer the reader to \cite[Chapter 1]{Lei14}.
In this section, we will define the categories of MAT-labeled graphs and LR-vines, and construct an explicit equivalence between these two categories.

 %********************************************************************************************************

\subsection{Equivalence of MAT-labeled graphs and LR-vines}
\label{subsec:equi-GnP}

\begin{definition}[Label-preserving homomorphism]
\label{def:lp-iso}
Let $(G,\lambda)$ and $(G', \lambda')$ be edge-labeled graphs. 
A \tbf{label-preserving homomorphism} from  $(G,\lambda)$ to $(G', \lambda')$, written $\sigma \colon (G,\lambda)  \longrightarrow (G', \lambda')$, is 
a map $\sigma \colon N_{G}  \longrightarrow N_{G'}$ such that for all $u, v \in N_G$, $\{u, v\} \in E_{G}$ implies $\{\sigma(u), \sigma(v)\} \in E_{G'}$  and $ \lambda(u, v) =  \lambda'(\sigma(u), \sigma(v))$. 
 
We call $\sigma$ an \tbf{isomorphism} if $\sigma$ is bijective and its inverse is a label-preserving homomorphism. 
The edge-labeled graphs $(G,\lambda)$ and $(G', \lambda')$ are said to be \tbf{isomorphic}, written $(G,\lambda) \simeq (G', \lambda')$ if there exists an isomorphism $\sigma  \colon  (G,\lambda)  \longrightarrow (G', \lambda')$. 
If $(G,\lambda) \simeq (G, \lambda')$, we say that two labelings $\lambda$ and $ \lambda'$ are the same (or isomorphic).
\end{definition}
If  $(G,\lambda) \simeq (G', \lambda')$ and $(G,\lambda)$ is an MAT-labeled graph, then $(G', \lambda')$ is also an MAT-labeled graph.

\begin{definition}[Category of MAT-labeled (complete) graphs]
\label{def:category-MG}
 The \tbf{category $\mathsf{MG}$ of MAT-labeled graphs} is the category whose objects are the MAT-labeled graphs and whose morphisms are the label-preserving homomorphisms. 
  The \tbf{category $\mathsf{MCG}$ of MAT-labeled complete graphs} is a full subcategory of $\mathsf{MG}$ whose objects are the MAT-labeled complete graphs.
\end{definition}

Recall the definition of rank and join-preserving homomorphisms of graded posets from Definition \ref{def:poset-iso}.

\begin{definition}[Category of (L)R-vines]
\label{def:category-LR}
 The \tbf{category $\mathsf{LRV}$ of LR-vines} is the category whose objects are the LR-vines and whose morphisms are the  homomorphisms preserving rank and join. 
   The \tbf{category $\mathsf{RV}$ of R-vines} is a full subcategory of $\mathsf{LRV}$ whose objects are the R-vines.
\end{definition}

First we need some lemmas. 

			\begin{lemma}
	\label{lem:induced-iso} 
	Let $\varphi \colon \Pc  \longrightarrow \Pc'$ be a rank-preserving homomorphism of LR-vines. 
	Suppose $\varphi$ preserves join of minimal pairs, i.e.~ if $x, y \in \min(\Pc)$ such that the join $ x\vee y$ exists, then $ \varphi (x) \vee \varphi(y)$ exists and $\varphi( x\vee y) = \varphi (x) \vee \varphi(y)$. 
Then $\varphi $ induces an isomorphism $ \Pc_{\le v}  \simeq \Pc'_{\le \varphi(v)}$ for every $v\in\Pc$.	
	\end{lemma}
	
		\begin{proof} 	
		Clearly,  $\varphi $ induces a homomorphism  $\varphi|_{ \Pc_{\le v}} \colon \Pc_{\le v}  \longrightarrow \Pc'_{\le \varphi(v)}$. 
		In particular, $\varphi(U_v) \subseteq U_{\varphi(v)}$. 
		Note that for any distinct $i, j \in U_v$, the join $i \vee j$ exists in the R-vine $\Pc_{\le v}$ by Remark \ref{rem:Rvine-only1}. 
			Since $\varphi$ preserves rank, $\rk'( \varphi ( i \vee j) ) = \rk( i \vee j) > \rk( i)$. 
	Since $\varphi$ preserves join of minimal pairs, $ \varphi (i) \vee \varphi(j)$ exists and $\varphi( i\vee j) = \varphi (i) \vee \varphi(j)$. 
	In particular,  $\varphi (i) \ne \varphi(j)$.
		Hence the elements in $\varphi(U_v)$ are pairwise distinct.
Thus  $\varphi(U_v) = U_{\varphi(v)}$ since  $|\varphi(U_v) | = |U_v  |= \rk(v) = \rk'(\varphi(v)) = |U_{\varphi(v)}  |$. 

Let $a,a' \in  \Pc_{\le v}  $ be such that $\varphi(a) = \varphi(a')$. 
We may write $a = i \vee j$ and $a' = i' \vee j'$ for minimal nodes $i \ne j$,   $i' \ne j'$. 
Thus, $ \varphi (i) \vee \varphi(j) =  \varphi (i') \vee \varphi(j')$. 
Again by Remark \ref{rem:Rvine-only1},  $\{\varphi (i), \varphi(j)\} =\{\varphi (i'), \varphi(j')\} $. 
Since the elements in $\varphi(U_v)$ are pairwise distinct, $\{ i , j\} = \{i' ,j'\}$. 
Hence $a=a'$. 
This implies that $\varphi|_{ \Pc_{\le v}}$ is injective. 
Moreover, $\left | \Pc_{\le v}  \right|=\left |  \Pc'_{\le \varphi(v)} \right|$ by Remark \ref{rem:card-R-vine}.
Hence $\varphi|_{ \Pc_{\le v}}$ is bijective. 

Now let $a,b \in  \Pc_{\le v}  $ be such that $\varphi(a) \le \varphi(b)$. 
Therefore, $U_{\varphi(a)} \subseteq U_{\varphi(b)}$ hence $\varphi(U_a)  \subseteq \varphi(U_b)$. 
It follows that $U_a  \subseteq U_b$. 
By Corollary \ref{cor:order-preserving}, $a \le b$.
	Thus the inverse of $\varphi|_{ \Pc_{\le v}}$ is a homomorphism.
	We conclude that   $\varphi|_{ \Pc_{\le v}}$ is an isomorphism.

	\end{proof}

				\begin{lemma}
	\label{lem:weak=all} 
	Let $\varphi \colon \Pc  \longrightarrow \Pc'$ be a rank-preserving homomorphism of LR-vines such that $\varphi$ preserves join of minimal pairs (Lemma \ref{lem:induced-iso}). 
Then $\varphi$ is join-preserving. 
	\end{lemma}
	
		\begin{proof} 	
 Let $a,b \in  \Pc $ and suppose the join $a \vee b$ exists in $\Pc$. 
 Write $v =a \vee b$. 
 Note that $\varphi(a), \varphi(b) \in  \Pc'_{\le \varphi(v)}$.
By Lemma \ref{lem:induced-iso}, $\varphi|_{ \Pc_{\le v}} \colon \Pc_{\le v}  \longrightarrow \Pc'_{\le \varphi(v)}$ is a poset isomorphism. 
It follows that $\varphi|_{ \Pc_{\le v}}$ is join-preserving. 
 Therefore, $\varphi (a) \vee \varphi(b)$ exists in $  \Pc'_{\le \varphi(v)}$ hence in $\Pc'$ and $\varphi( a\vee b) = \varphi (a) \vee \varphi(b)$. 
 Thus $\varphi$ is join-preserving. 
	\end{proof}

		\begin{lemma}
	\label{lem:assoc-morphism2} 
Let $\Pc $ and $ \Pc'$ be LR-vines and suppose there is a homomorphism $\varphi \colon \Pc  \longrightarrow \Pc'$ preserving rank and join. 
Let $(G,\lambda)$ (resp.~$(G', \lambda')$) denote the MAT-labeled graph defined by $\Pc $ (resp.~$\Pc'$)  from Definition \ref{def:P-to-G} and Theorem \ref{thm:poset-to-MAT}. 
Define a map $\Omega(\varphi )\colon N_{G}  \longrightarrow N_{G'}$ by sending
each node    $i$ in $G$ to a node $\varphi(i)$ in $G'$ for  $1 \le i \le \ell$. 
Then   $\Omega(\varphi )$  is a label-preserving homomorphism from $(G,\lambda)$ to $(G', \lambda')$.
	\end{lemma}
	
		\begin{proof} 	
	Let $\{i, j\} \in E_{G}$  be an edge in $G$ for distinct nodes $i, j \in N_G$. 
	Then the join $ i \vee j$  exists in $\Pc$. 
	Since $\varphi$ is  join-preserving, $ \varphi (i) \vee \varphi(j)$ exists and $\varphi( i\vee j) = \varphi (i) \vee \varphi(j)$. 
	Since $\varphi$ is  rank-preserving, $\varphi (i) \ne \varphi(j)$.
	Therefore, $\{\varphi (i), \varphi(j)\} \in E_{G'}$. 
	Furthermore,   
	$$ \lambda'(\varphi (i), \varphi(j)) =\rk'( \varphi (i) \vee \varphi(j) ) -1= \rk( i \vee j)-1=  \lambda(i,j).$$
	Thus   $\Omega(\varphi )$  is a label-preserving homomorphism.
 
	\end{proof}

	\begin{lemma}
	\label{lem:assoc-morphism1} 
Let $(G,\lambda)$ and $(G', \lambda')$ be MAT-labeled graphs and suppose there is a label-preserving homomorphism $\sigma \colon (G,\lambda)  \longrightarrow (G', \lambda')$. 
Let $\Pc $ (resp.~$\Pc'$) denote the LR-vine defined by $(G,\lambda)  $ (resp.~$(G', \lambda')$)  from Definition \ref{def:G-to-P} and Theorem \ref{thm:MAT-to-poset}. 
Define  a map $\Psi(\sigma) \colon \Pc  \longrightarrow \Pc'$ by sending $v\in \Pc$ to $\{\sigma(a) \mid a \in v\}\in \Pc'$.
Then   $\Psi(\sigma) $ is a  homomorphism preserving rank and join between $\Pc $ and $ \Pc'$.
	\end{lemma}
	
			\begin{proof} 	
 Let $\varphi  =\Psi(\sigma)  $. 
Note that since $\varphi$ is   label-preserving, if $v = K_e \in \Pc$ where $K_e$ is a principal clique in $(G,\lambda)$ for $e \in E_G$, then $\sigma(K_e)=\{\sigma(a) \mid a \in K_e\}=K_{\sigma(e) }$ is a principal clique in $(G', \lambda')$. 
Hence $\varphi $ is indeed well-defined. 
It is also easily seen that $\varphi $ is a rank-preserving  homomorphism since $\lambda(e) = |K_e|-1$. 
Let $\{i\} , \{j\} \in \Pc$ be distinct minimal nodes in $\Pc$ such that $\{i\} \vee \{j\}$ exists in $\Pc$. 
By Corollary \ref{cor:cond-sets}, we may write  $K_e=\{i\} \vee \{j\} \in \Pc$ for $e=\{i,j\} \in E_G$. 
Also by Corollary \ref{cor:cond-sets}, $ K_{\sigma(e) }=\{ \sigma(i) \} \vee \{ \sigma(j)  \} \in \Pc'$ since  $\sigma(e)=\{\sigma(i) ,\sigma(j)\} \in E_{G'}$. 
Therefore,  $\varphi(  \{i\} \vee \{j\} ) = \varphi (\{i\}) \vee \varphi(\{j\})$. 
Thus $\varphi  $  preserves join of minimal pairs.
 By Lemma \ref{lem:weak=all}, $\varphi  $ is   join-preserving.

	\end{proof}

	The following two lemmas give a construction of functors  between $\MG$ and $\LRV $.
The proofs   are routine.
 
	\begin{lemma}
	\label{lem:functor2} 
Define a mapping  $\Omega \colon \LRV\longrightarrow  \MG$ by associating

\begin{enumerate}[(1)]
		\item   each object $ \Pc$ in $\LRV $ to an object 
$\Omega(\Pc) = (G,\lambda)$ in $\MG$, where $ (G,\lambda)$ is the MAT-labeled graph defined by $\Pc $  from Definition \ref{def:P-to-G} and Theorem \ref{thm:poset-to-MAT}, and 
		\item  each morphism $\varphi \colon \Pc  \longrightarrow \Pc'$ in $\LRV $ to a morphism   $\Omega(\varphi ) \colon \Omega(\Pc)  \longrightarrow\Omega(\Pc')$ in $\MG$, where $\Omega(\varphi ) $ is the label-preserving homomorphism  from Lemma \ref{lem:assoc-morphism2}.
	\end{enumerate}
Then   $\Omega $ is a functor from $\LRV $ to $\MG$.
	\end{lemma}
	
		\begin{lemma}
	\label{lem:functor1} 
Define a mapping $\Psi\colon\MG \longrightarrow \LRV $ by associating
\begin{enumerate}[(1)]
		\item  each object $(G,\lambda)$ in $\MG$ to an object 
$\Psi(G,\lambda) =\Pc$ in $\LRV $, where $\Pc $ is the LR-vine defined by $(G,\lambda)  $   from Definition \ref{def:G-to-P} and Theorem \ref{thm:MAT-to-poset}, and
		\item  each morphism $\sigma \colon (G,\lambda)  \longrightarrow (G', \lambda')$ in $\MG$ to a morphism $\Psi(\sigma) \colon \Psi(G,\lambda)  \longrightarrow \Psi(G', \lambda')$  in $\LRV $, where $\Psi(\sigma) $ is the  homomorphism preserving rank and join  from Lemma \ref{lem:assoc-morphism1}.
	\end{enumerate}
Then   $\Psi $ is a functor from $\MG$ to $\LRV $.
	\end{lemma}
 	
  We are now ready to prove the main result of the paper. 
  
		\begin{theorem} 
	\label{thm:1-to-1}
The composite functor $\Psi  \Omega$ (resp.~$\Omega  \Psi$) is naturally isomorphic to the identity functor $ 1_{\LRV}$ (resp.~$1_{\MG}$). 
As a result, the categories $\MG$ and $\LRV $ are equivalent. 
	\end{theorem}
		\begin{proof}
		For every  LR-vine $\Pc$, recall  from Proposition \ref{prop:Phat} the  LR-vine $ \widehat\Pc =\{ U_v \mid v \in \Pc\}$ and the isomorphism 
	$$\eta_\Pc\colon \Pc  \longrightarrow \widehat\Pc  \quad \text{via} \quad v \mapsto U_v.$$ 

		By Lemmas \ref{lem:functor1} and \ref{lem:functor2}, the  functor $\Psi  \Omega \colon \LRV\longrightarrow  \LRV$ assigns 
\begin{enumerate}[(1)]
		\item   each LR-vine $ \Pc$  to the LR-vine $\Psi \Omega(\Pc)= \widehat\Pc$, and 
		\item  each morphism $\varphi \colon \Pc  \longrightarrow \Pc'$ in $\LRV $ to a morphism   $\Psi \Omega(\varphi ) \colon \widehat\Pc \longrightarrow \widehat{\Pc'}$ in $\LRV $ defined by sending $U_v \in  \widehat\Pc$ to $\varphi(U_v)=U_{\varphi(v)} \in  \widehat{\Pc'}$ for every $ v\in \Pc$. 
		The equality holds by Lemma \ref{lem:induced-iso}.
	\end{enumerate}	
	
Similarly, the  functor $ \Omega\Psi  \colon \MG\longrightarrow  \MG$ assigns 
\begin{enumerate}[(1')]
		\item   each MAT-labeled graph $ (G,\lambda)$  to an MAT-labeled graph $ \Omega\Psi(G,\lambda)= (\widehat G,\widehat\lambda)$, where $N_{\widehat G} = \{ \{i\} \mid i \in N_G\}$, $E_{\widehat G} = \{ \{\{i\},\{j\}\} \mid \{i,j\} \in E_G\}$, and $\widehat\lambda(\{i\},\{j\}) =\lambda(i,j)$ (see also Corollary \ref{cor:cond-sets}), and
		\item  each morphism $\sigma \colon (G,\lambda)  \longrightarrow (G', \lambda')$ in $\MG$ to a morphism $ \Omega\Psi(\sigma) \colon   (\widehat G,\widehat\lambda) \longrightarrow   (\widehat {G'},\widehat{\lambda'})$  in $\MG$ defined by sending $\{i\}  \in N_{\widehat G}$ to $\{ \sigma(i)  \} \in N_{\widehat {G'}}$ for every $ i \in N_G$.		
		
	\end{enumerate}			
	
	First we prove $\Psi  \Omega \simeq  1_{\LRV}$, i.e.~$\Psi  \Omega$   is naturally isomorphic to  $ 1_{\LRV}$. 	
	For every morphism $\varphi \colon \Pc  \longrightarrow \Pc'$ in $\LRV $, we have a commutative diagram in Figure \ref{fig:psiomega}.	 

	\begin{figure*}[htbp]
\centering
\begin{subfigure}{.25\textwidth}
		\begin{tikzcd}
\Pc \arrow{r}{\eta_\Pc} \arrow[swap]{d}{\varphi } & \widehat\Pc \arrow{d}{\Psi \Omega(\varphi )} \\%
\Pc'  \arrow{r}{\eta_{\Pc'}}& \widehat{\Pc'}
\end{tikzcd}
\end{subfigure}%
\quad
\begin{subfigure}{.35\textwidth}
  \centering
 		\begin{tikzcd}
v \arrow[mapsto]{r}{\eta_\Pc} \arrow[mapsto, swap]{d}{\varphi } &U_v \arrow[mapsto]{d}{\Psi \Omega(\varphi )} \\%
\varphi(v)   \arrow[mapsto]{r}{\eta_{\Pc'}}& U_{\varphi(v)} =\varphi(U_v) 
\end{tikzcd}
 \end{subfigure}
\caption{\small Commutative diagram that shows $\Psi  \Omega \simeq  1_{\LRV}$.}
\label{fig:psiomega}
\end{figure*}

This follows that $\eta\colon  1_{\LRV}  \longrightarrow\Psi  \Omega$ with component $\eta_\Pc$ at $\Pc$ is a natural isomorphism.
Thus  $\Psi  \Omega$   is naturally isomorphic to  $ 1_{\LRV}$.  

The proof for $\Omega  \Psi \simeq 1_{\MG}$ is similar and easier. 
For every object $(G,\lambda)  $ in $\MG$, the following map is an isomorphism
	$$\epsilon_{(G,\lambda) }\colon (G,\lambda)   \longrightarrow  (\widehat G,\widehat\lambda)  \quad \text{via} \quad N_G \ni i \mapsto \{i\}  \in N_{\widehat G}.$$
	Furthermore,   $  \Omega\Psi$   is naturally isomorphic to  $ 1_{\MG}$ via the natural isomorphism $\epsilon \colon  1_{\MG}  \longrightarrow  \Omega\Psi  $ with component $\epsilon_{(G,\lambda) }$ at $(G,\lambda)$. 

 	\end{proof}

	The following corollary is straightforward from Theorem \ref{thm:1-to-1}. 
	
	\begin{corollary}
\label{cor:1-to-1-restrict}
The restriction $\Psi |_{\MCG}$ (resp.~$\Omega |_{\RV}$) is a functor from $\MCG$ (resp.~$ \RV$) to $ \RV$ (resp.~$\MCG$). 
Furthermore, the composite functor $\Psi  |_{\MCG} \Omega |_{\RV}$ (resp.~$\Omega  |_{\RV} \Psi  |_{\MCG}$) is naturally isomorphic to the identity functor $ 1_{\RV}$ (resp.~$1_{\MCG}$). 
As a result, the categories $\MCG$ and $\RV $ are equivalent. 
\end{corollary}

Another interesting consequence is the existence of a pushout in the category $\MG$ hence in $\LRV$ owing to the gluing method in  Lemma \ref{lem:glue}.

\begin{corollary}
\label{cor:pushout}
Suppose we are in the situation of Lemma \ref{lem:glue}.
Denote by $\mu_1\colon (G', \lambda')\hooklongrightarrow (G_1, \lambda_1)$, $\mu_2\colon(G', \lambda')\hooklongrightarrow (G_2, \lambda_2)$, $\sigma_1\colon (G_1, \lambda_1)\hooklongrightarrow (G, \lambda)$ and $\sigma_2\colon (G_2, \lambda_2)\hooklongrightarrow (G, \lambda)$ the embeddings (in particular, morphisms in $\MG$) of the MAT-labeled graphs. 
Then $((G, \lambda), \sigma_1, \sigma_2)$ is a pushout of  $ \mu_1$ and  $\mu_2$.

\end{corollary}
 
		\begin{proof} 	
		It is easily seen that $\sigma_1 \mu_1 = \sigma_2 \mu_2$. 
		Hence the square diagram commutes. 
		Now given another triple $((G_3, \lambda_3), \sigma'_1, \sigma'_2)$ with  $\sigma'_1 \mu_1 = \sigma'_2 \mu_2$, let $\theta\colon (G, \lambda)\longrightarrow (G_3, \lambda_3)$ be defined by $\theta(u) = \sigma'_1(u)$ for $u \in N_{G_1}$,  $\theta(u) = \sigma'_2(u)$ for $u \in N_{G_2}$. 
		Thus $\theta\sigma_1 =  \sigma'_1$ and $\theta\sigma_2 =  \sigma'_2$. 
		It is also easy to see that $\theta$ is the  unique morphism making the diagram commute. 
		$$
\begin{tikzcd}
       (G', \lambda') \ar[r, hook, "\mu_1"] \ar[d,hook, "\mu_2"]  &  (G_1, \lambda_1) \ar[d, hook, "\sigma_1"]  \ar[ddr, bend left, "\sigma'_1"]  & \\
       (G_2, \lambda_2) \ar[r,hook, "\sigma_2"] \ar[drr, bend right, "\sigma'_2"]  &   (G, \lambda) \ar[dr, dashed, "\exists!\, \theta"]  &  \\
      &  &   (G_3, \lambda_3)
\end{tikzcd}
$$

We conclude that $((G, \lambda), \sigma_1, \sigma_2)$ is a pushout of  $ \mu_1$ and  $\mu_2$.
	\end{proof}

	 %********************************************************************************************************

\subsection{Equivalence of LR-vines and m-vines}
\label{subsec:equi-PnmV}

		\begin{theorem} 
	\label{thm:LR-characterize}
Let $\Pc$ be a vine. 
The following are equivalent:
\begin{enumerate}[(1)]
		\item  $\Pc$ is an m-vine, i.e.~by definition, $\Pc$ is an ideal of an R-vine.
		\item $\Pc$ satisfies the proximity condition.
		\item $\Pc$ is an LR-vine.
 
	\end{enumerate}
	
	\end{theorem}

		\begin{proof}
$(2) \Leftrightarrow (3)$ is shown in Proposition \ref{prop:LR=P}. 
$(1) \Rightarrow (3)$ is straightforward from Remark \ref{rem:ideal-LRvine}. 
It remains to show $(1) \Leftarrow (3)$. 
The proof is based on the following diagram:
$$
\begin{tikzcd}[row sep=scriptsize,  column sep=normal]
& \widehat\Pc \ar[dd, hook, dashed, swap, "\iota'"] 
&
 \\
 \Pc \ar[ur, dashed, "\eta_\Pc"] \ar[rr, crossing over, "\Omega" near start] \ar[dd, hook, crossing over, "\iota^\prime"']
    &
      & (G,\lambda) \ar[dd, hook, crossing over, "\iota"']  \ar[ul, dashed, swap, "\Psi"]  \\
&\widehat  \scR \ar[dl, swap, "\eta_\scR^{-1}"]  & \\
 \scR &&  (K,\widetilde\lambda) \ar[ul, swap, "\Psi"] \ar[ll, crossing over, dashed] 
\end{tikzcd}
$$
 \vskip .5em
 First given an LR-vine $\Pc$, let $ (G,\lambda) = \Omega(\Pc) $ be the MAT-labeled graph associated to $\Pc$ via the functor $\Omega$ from Lemma \ref{lem:functor2}. 
 Next let $(K,\widetilde\lambda)$ be the MAT-labeled complete graph such that $ (G, \lambda) \le (K,\widetilde\lambda) $ from Proposition \ref{prop:extend}. 
 In particular, there exists an embedding  $\iota \colon (G, \lambda)\hooklongrightarrow (K,\widetilde\lambda) $. 
 Then let $\Psi |_{\MCG} (K,\widetilde\lambda)= \widehat \scR $ be the R-vine associated to $(K,\widetilde\lambda)$ via the functor $\Psi |_{\MCG}$ from Corollary \ref{cor:1-to-1-restrict}. 
 Finally, let $\scR$ be the R-vine isomorphic to $ \widehat \scR $ via the (inverse of) poset isomorphism $\eta_\scR$ from Proposition \ref{prop:Phat}. 
 
$(1) \Leftarrow (3)$ is proved once we show that $\Pc$ is an ideal of $\scR$. 
Indeed, by the construction, $\Pc$ is an induced subposet of $\scR$. 
(One may see this via the sequence $\Pc \longrightarrow  \widehat\Pc  \hooklongrightarrow \widehat  \scR \longrightarrow   \scR $ in the diagram.) 
Note also that $\Pc$ and $\scR$ have the same set of minimal nodes. 
Let  $w,v \in \scR$ be two nodes with $w \le v$ and $v \in \Pc$. 
We need to show that $w \in \Pc$. 
The assertion follows easily if either $w$ or $v$ is a minimal node. 
We may assume that both $w$ and $v$ are not minimal. 
Since $C_w \subseteq U_w \subseteq U_v$ in $\scR$, and $U_v$ in $\Pc$ is the same as $U_v$ in $\scR$, we have $C_w \subseteq U_v$ in $\Pc$. 
By Remark \ref{rem:Rvine-only1}, there exists a non-minimal node $w'$ in the R-vine $\Pc_{\le v}$ such that $C_w =C_{w'}$.
Apply Corollary \ref{cor:C=C} for two non-minimal nodes $w,w'$ in $\scR$, we have $w=w'$. 
Hence $w $ is an element  in $\Pc$, as desired. 
	\end{proof}

%********************************************************************************************************

\section{Applications}
\label{sec:appl}

  %********************************************************************************************************
 
\subsection{From LR-vines to  MAT-labeled graphs}
\label{subsec:appl-vine-to-MAT}

    %********************************************************************************************************

\subsubsection{A poset characterization of MAT-free graphic arrangements}
\label{subsubsec:characterize}

The most important application of our main result (Theorem \ref{thm:1-to-1}) is an affirmative answer for the question of Cuntz-M{\"u}cksch  (Question \ref{ques:CM}) in the case of graphic arrangements: MAT-free graphic arrangements have a poset characterization by LR-vines. 
(Note that LR-vines generalize the root poset of type $A$ by Remark \ref{rem:D-vine-RP}.)
We give below two examples to illustrate the correspondence. 
First we need a definition.

 	\begin{definition}[C-vine] 
\label{def:C-vine}
		An R-vine is called a  \tbf{C-Vine} if each associated tree has the largest possible number of vertices of degree $1$.
		Equivalently, each associated tree is a \emph{star graph}. 
\end{definition}

D-vine and C-vine can be regarded as the ``extreme" cases of R-vines.

		\begin{example} 
			\label{ex:C-vine} 	
			In dimension $4$, there are exactly two non-isomorphic R-vine structures: D-vine and C-vine. 
			Likewise, there are exactly two non-isomorphic MAT-labeled complete graphs on $4$ vertices. 
			Figure~\ref{fig:C-vine} depicts a C-vine on $[4]$ (top right), the associated forests  (top left), the associated MAT-labeled complete graph (bottom right) via  the functor $\Omega$ from Lemma \ref{lem:functor2}, and the MAT-partition  (bottom left) of the corresponding graphic arrangement. 
			The C-vine in dimension $\ge 4$ is not an ideal of any D-vine hence the corresponding MAT-partition is not obtained from an ideal of the type $A$ root poset.
	\end{example} 

\begin{figure}[htbp!]
\begin{subfigure}{.5\textwidth}
   \centering
\begin{tikzpicture}[scale=.8]   
\draw (0,4) node[v](34){} node [above]{\tiny\textcolor{blue}{$34|12$}};
  \draw (1,3) node[v](24){} node [above right]{\tiny\textcolor{red}{$24|1$}};
  \draw (-1,3) node[v](23){} node [above left]{\tiny\textcolor{red}{$23|1$}};
  \draw (0,2) node[v](12){} node [above]{\tiny $12$};
    \draw (-2,2) node[v](13){} node [above]{\tiny $13$};
    \draw (2,2) node[v](14){} node [above]{\tiny $14$};    
      \draw (0,0) node[v](1){} node [above right]{\tiny$1$};
      \draw (0,1) node[v](2){} node [right]{\tiny$2$};
      \draw (-0.865,-0.5)  node[v](3){} node [left]{\tiny$3$};
      \draw  ( 0.865,-0.5)   node[v](4){} node [right]{\tiny$4$};   
  \draw (24) -- (23) ;
       \draw (13)-- (12) -- (14) ;
    \draw (1)--(3)  ;
      \draw (1)--(4)  ;
               \draw (1)--(2)  ;
\end{tikzpicture}
\caption*{Associated trees}
\end{subfigure}%
\begin{subfigure}{.5\textwidth}
   \centering
\begin{tikzpicture}[scale=1]
 
  \node (34) at (-1,3) {\textcolor{blue}{$34|12$}};
  \node (24) at (0,2) {\textcolor{red}{$24|1$}};
  \node (23) at (-2,2) {\textcolor{red}{$23|1$}};
  \node (14) at (1,1) {$14$};
  \node (13) at (-1,1) {$13$};
   \node (12) at (-3,1) {$12$};
         \node (4) at (2,0) {$4$};
       \node (3) at (0,0) {$3$};
     \node (2) at (-2,0) {$2$};
   \node (1) at (-4,0) {$1$};
   
  \draw (1)--(12) -- (23) -- (34) ;
    \draw (4)--(14) -- (24) -- (34) ;
    \draw  (3)--(13) -- (23) ;
        \draw (2)-- (12) -- (24) ;
          \draw (1)--(13)  ;
           \draw (1)--(14)  ;
\end{tikzpicture}
\caption*{C-vine}
\end{subfigure}%
\vskip 1em

\begin{subfigure}{.5\textwidth}
  \centering
\begin{tikzpicture}[scale=1]
  \node (34) at (-1,3) {\textcolor{blue}{$x_3-x_4$}};
  \node (24) at (0,2) {\textcolor{red}{$x_2-x_4$}};
  \node (23) at (-2,2) {\textcolor{red}{$x_2-x_3$}};
  \node (14) at (1,1) {$x_1-x_4$};
  \node (13) at (-1,1) {$x_1-x_3$};
   \node (12) at (-3,1) {$x_1-x_2$};
   
   \draw (12) -- (23) -- (34) --(24)--(14)  ;
    \draw  (13) -- (23) ;
        \draw (12) -- (24) ;
\end{tikzpicture}
  \caption*{MAT-partition}
\end{subfigure}%
\begin{subfigure}{.5\textwidth}
  \centering
\begin{tikzpicture}[scale=1.2]
\begin{scope}[every node/.style={circle,thick,draw, inner sep=1.7pt, fill=blue!20}]
    \node (v1) at (0,0) {$v_1$};
    \node (v2) at (2,0) {$v_2$};
    \node (v4) at (2,2) {$v_4$};
    \node (v3) at (0,2) {$v_3$};
\end{scope}
\draw[red, thick] (v3)--(v2) node [near end, above, red] {$2$};
\draw[red, thick] (v4)--(v2) node [midway, right, red] {$2$};
\draw[blue, thick] (v4)--(v3) node [midway, below, blue] {$3$};
\draw[thick] (v4)--(v1) node [near end, above] {$1$};
\draw[thick] (v2)--(v1) node [midway, below] {$1$};
\draw[thick] (v3)--(v1) node [midway, left] {$1$};
 
\end{tikzpicture}
  \caption*{MAT-labeled graph}
\end{subfigure}
\caption{\small C-vine on $4$ elements, the associated trees, MAT-labeled complete graph, and MAT-partition from Example \ref{ex:C-vine}.}
\label{fig:C-vine}
\end{figure}
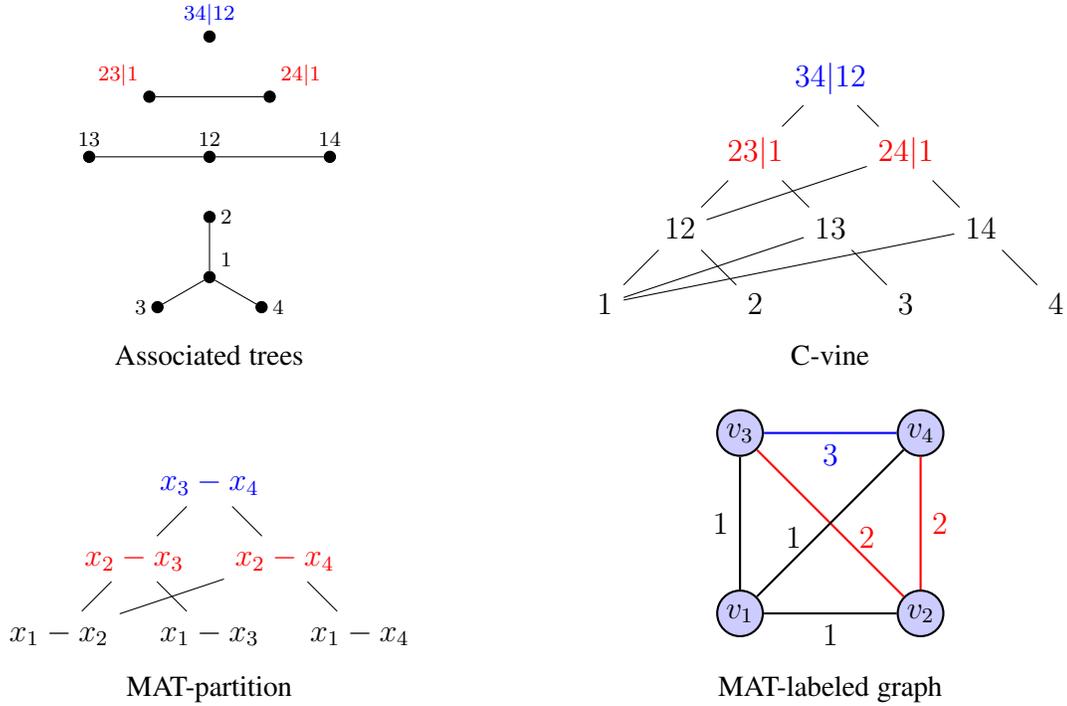

		\begin{example} 
			\label{ex:LRV} 	
			Figure~\ref{fig:LRV} depicts  on the right an MAT-labeled graph $ (G, \lambda)$ on $5$ vertices and an MAT-labeled complete graph $(K,\widetilde\lambda)$ such that $ (G, \lambda) \le (K,\widetilde\lambda) $. 
			The complementary edges are shown in dashed lines. 
			The graphs $ (G, \lambda)$ and $(K,\widetilde\lambda)$ correspond (via the functor  $  \Psi  $ from Lemma \ref{lem:functor1}) to the LR-vine $\Pc$ and R-vine $\scR$ on the left, respectively. 
			In this case, $\Pc$ is the $3$-lower truncation of $\scR$. 
 
	\end{example}

\begin{figure}[htbp!]
\begin{subfigure}{.57\textwidth}
   \centering
\begin{tikzpicture}[scale=1]

  \node (x) at (-4,2.5) {};
    \node (y) at (4,2.5){};
 
    \node (15) at (0,4) {\textcolor{brown}{\dbox{$15|234$}}};
    
   \node (25) at (1,3) {\textcolor{blue}{\dbox{$25|34$}}};
  \node (12) at (-1,3) {\textcolor{blue}{\dbox{$12|34$}}};
  
    \node (35) at (2,2) {\textcolor{red}{$35|4$}};
  \node (23) at (0,2) {\textcolor{red}{$23|4$}};
  \node (13) at (-2,2) {\textcolor{red}{$13|4$}};
  
  \node (45) at (3,1) {$45$};  
  \node (34) at (1,1) {$34$};
  \node (24) at (-1,1) {$24$};
   \node (14) at (-3,1) {$14$};
   
      \node (5) at (4,0) {$5$};
         \node (4) at (2,0) {$4$};
       \node (3) at (0,0) {$3$};
     \node (2) at (-2,0) {$2$};
   \node (1) at (-4,0) {$1$};
   
  \draw (1)--(14) -- (13) ;
      \draw [dashed]  (13) -- (12) ;
            \draw [dashed]  (23) -- (12) ;
    \draw [dashed]  (12) -- (15) ;
    \draw (2)--(24) -- (23);
        \draw [dashed]  (23) -- (25) ;
    \draw  (3)--(34)--(35) ;
        \draw  (4) --(45) ;
            \draw  (5)--(45)--(35) ;
                    \draw [dashed]  (35) -- (25) -- (15) ; 
  \draw (4)--(34) -- (13)   ;
    \draw (4)--(14)    ;
      \draw (4)--(24)    ;
            \draw (34)--(23)   ;
              \draw[ForestGreen, thick] (x) -- (y) ;
 
\end{tikzpicture}
\end{subfigure}%
\begin{subfigure}{.43\textwidth}
  \centering
\begin{tikzpicture}[scale=1.2]
\begin{scope}[every node/.style={circle,thick,draw, inner sep=1.7pt, fill=blue!20}]
 
        \node (v3) at (0,0) {$v_3$};
    \node (v4) at (2,0) {$v_4$};
    \node (v1) at (-1,1.7) {$v_1$};
    \node (v5) at (3,1.7) {$v_5$};
        \node (v2) at (1,3) {$v_2$};
        
\end{scope}
\draw[red, thick] (v3)--(v2) node [near end, left, red] {$2$};
\draw[thick] (v4)--(v2) node [near end, right] {$1$};
\draw[thick] (v4)--(v3) node [midway, below] {$1$};
\draw[thick] (v4)--(v1) node [near end, right, above] {$1$};
\draw[blue, dashed] (v2)--(v1) node [midway, above] {$3$};
\draw[blue,   dashed] (v5)--(v2) node [midway, above] {$3$};
\draw[brown,  dashed] (v5)--(v1) node [midway, above] {$4$};
\draw[red, thick] (v3)--(v1) node [midway, left, red] {$2$};
\draw[red, thick] (v3)--(v5) node [near end, above, red] {$2$};
\draw[thick] (v4)--(v5) node [midway, right] {$1$};
\end{tikzpicture}
\end{subfigure}
\caption{\small An MAT-labeled   graph on $5$ vertices (right) and the LR-vine $\Pc = \Psi(G,\lambda) $   (left) from Lemma \ref{lem:functor1} with its $3$-lower truncation.}
\label{fig:LRV}
\end{figure}
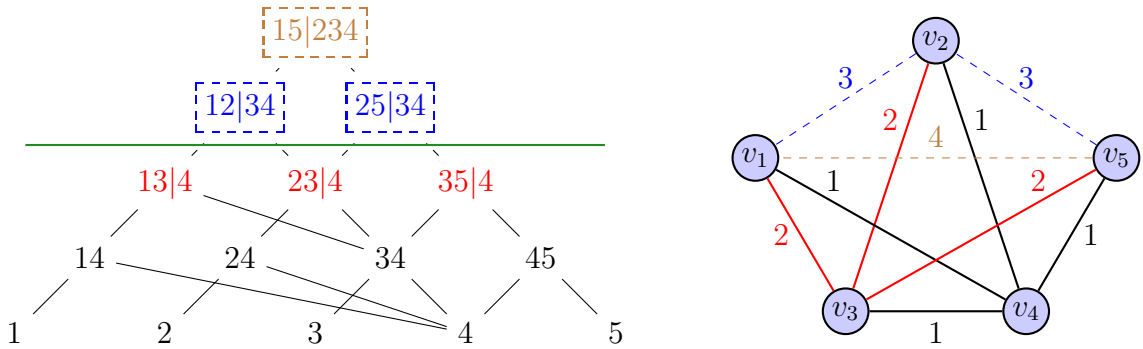

  %********************************************************************************************************

\subsubsection{Number of non-isomorphic MAT-labelings of complete graphs}
\label{subsubsec:MATCG}

The number of equivalence classes of R-vines in dimension $\ell$ is given in   \cite[\S 10.3]{KJ11}. 
By our Corollary \ref{cor:1-to-1-restrict}, this number is equal to the number of non-isomorphic MAT-labelings of the complete graph on $\ell$ vertices. 
We immediately have the following:
 		\begin{theorem}
	\label{thm:no.non-iso} 
The number $E_\ell$ of non-isomorphic MAT-labelings of the complete graph $K_\ell$ for $\ell \ge 1$ is given by $E_1 =E_2=E_3=1$ and $E_\ell=(A_\ell+B_\ell)/2$ for $\ell \ge 4$, where
$$A_\ell = 2^{(\ell-2)(\ell-3)/2}, \quad B_\ell = \sum_{k=1}^{\lfloor \ell/2 \rfloor -1}A_\ell c_k 2^{-k + \sum_{i=0}^{k-1} (\ell-4-2i)},$$
and $c_k = 1$ for all $k$ except $c_{\lfloor \ell/2 \rfloor -1} = 2$. 
	\end{theorem}
	
	The first $8$ elements of the sequence $(E_\ell)$ are $1,1  ,1 , 2 , 6,40, 560, 17024$ mentioned in \S\ref{subsec:mov}.
In particular, $E_4=2$ and these MAT-labelings are given in Figures \ref{fig:D-vine} and \ref{fig:C-vine}.

  %********************************************************************************************************

\subsubsection{Upper truncation of MAT-labeled graphs}
\label{subsubsec:UT}

In Remark \ref{rem:low-upp}, we discussed two ways to obtain a new LR-vine from a given LR-vine by  upper or lower truncation. 
From our main result  \ref{thm:1-to-1}, the lower truncation simply corresponds to deleting the edges of high label in the associated MAT-labeled graph (Figure~\ref{fig:LRV}). 
The upper truncation, however, gives rise to a nontrivial graph operation which produces an MAT-labeled graph from a given one. 
We shall not give an explicit formulation of the operation but instead illustrate it by an example in Figure \ref{fig:upper-truncation}. 

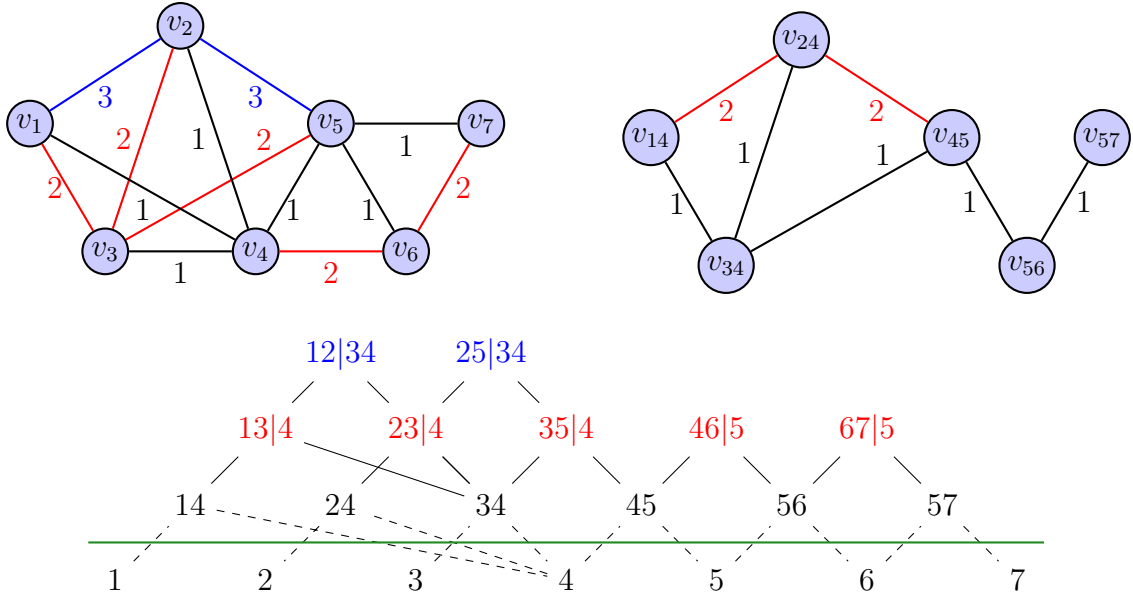
\begin{figure}[htbp!]
\begin{subfigure}{.5\textwidth}
  \centering
\begin{tikzpicture}[scale=1]
\begin{scope}[every node/.style={circle,thick,draw, inner sep=1.7pt, fill=blue!20}]
    \node (v3) at (0,0) {$v_3$};
    \node (v4) at (2,0) {$v_4$};
    \node (v6) at (4,0) {$v_6$};
    \node (v1) at (-1,1.7) {$v_1$};
    \node (v5) at (3,1.7) {$v_5$};
    \node (v7) at (5,1.7) {$v_7$};
        \node (v2) at (1,3) {$v_2$};
\end{scope}
\draw[red, thick] (v3)--(v1) node [near end, below, red] {$2$};
\draw[red, thick] (v3)--(v5) node [near end, above, red] {$2$};
\draw[red, thick] (v3)--(v2) node [midway, left, red] {$2$};
\draw[red, thick] (v4)--(v6) node [midway, below, red] {$2$};
\draw[red, thick] (v7)--(v6) node [midway, right, red] {$2$};

\draw[blue, thick] (v2)--(v1) node [midway, below, blue] {$3$};
\draw[blue, thick] (v2)--(v5) node [midway, below, blue] {$3$};

\draw[thick] (v4)--(v5) node [midway, below] {$1$};
 \draw[thick] (v6)--(v5) node [midway, below] {$1$};
  \draw[thick] (v7)--(v5) node [midway, below] {$1$};
  \draw[thick] (v4)--(v1) node [midway, below] {$1$};
  \draw[thick] (v4)--(v2) node [midway, left] {$1$};
  \draw[thick] (v4)--(v3) node [midway, below] {$1$};
 
\end{tikzpicture}
\end{subfigure}%
\begin{subfigure}{.5\textwidth}
  \centering
\begin{tikzpicture}[scale=1]
\begin{scope}[every node/.style={circle,thick,draw, inner sep=1.7pt, fill=blue!20}]
    \node (v3) at (0,0) {$v_{34}$};
    \node (v6) at (4,0) {$v_{56}$};
    \node (v1) at (-1,1.7) {$v_{14}$};
    \node (v5) at (3,1.7) {$v_{45}$};
    \node (v7) at (5,1.7) {$v_{57}$};
        \node (v2) at (1,3) {$v_{24}$};
\end{scope}
\draw[thick] (v3)--(v1) node [near end, below] {$1$};
\draw[thick] (v3)--(v5) node [near end, above] {$1$};
\draw[thick] (v3)--(v2) node [midway, left] {$1$};
 \draw[thick] (v7)--(v6) node [midway, right] {$1$};

\draw[red, thick] (v2)--(v1) node [midway, below] {$2$};
\draw[red, thick] (v2)--(v5) node [midway, below] {$2$};

  \draw[thick] (v6)--(v5) node [midway, left] {$1$};

\end{tikzpicture}
\end{subfigure}%
\vskip 1em

\begin{subfigure}{1\textwidth}
   \centering
\begin{tikzpicture}[scale=1]

  \node (x) at (-4.5,.5) {};
    \node (y) at (8.5,.5){};
  
     \node (25) at (1,3) {\textcolor{blue}{$25|34$}};
  \node (12) at (-1,3) {\textcolor{blue}{$12|34$}};
  
        \node (67) at (6,2) {\textcolor{red}{$67|5$}};
      \node (46) at (4,2) {\textcolor{red}{$46|5$}};
    \node (35) at (2,2) {\textcolor{red}{$35|4$}};
  \node (23) at (0,2) {\textcolor{red}{$23|4$}};
  \node (13) at (-2,2) {\textcolor{red}{$13|4$}};
  
    \node (57) at (7,1) {$57$};  
  \node (56) at (5,1) {$56$};  
  \node (45) at (3,1) {$45$};  
  \node (34) at (1,1) {$34$};
  \node (24) at (-1,1) {$24$};
   \node (14) at (-3,1) {$14$};
   
            \node (7) at (8,0) {$7$};
         \node (6) at (6,0) {$6$};
      \node (5) at (4,0) {$5$};
         \node (4) at (2,0) {$4$};
       \node (3) at (0,0) {$3$};
     \node (2) at (-2,0) {$2$};
   \node (1) at (-4,0) {$1$};
   
      \draw   (13) -- (12)--(23) -- (25)-- (35);
   \draw  (14) -- (13) --(34);
            \draw   (24) -- (23) -- (34)--(35)--(45)--(46)--(56) -- (67) -- (57);
                                         \draw [dashed]  (3)--(34)--(4) --(45) -- (5) -- (56) -- (6)-- (57) -- (7) ; 
          \draw [dashed] (1)--(14)--(4) --(24)--(2)   ;

            \draw (34)--(23)   ;
              \draw[ForestGreen, thick] (x) -- (y) ;
 
\end{tikzpicture}
\end{subfigure}
\caption{\small An MAT-labeled   graph on $7$ vertices (top left), the LR-vine $\Pc = \Psi(G,\lambda) $   (bottom) from Lemma \ref{lem:functor1} and its $1$-upper truncation $\overline{\Pc}_{\ge 1}$, and the MAT-labeled graph $\Omega(\overline{\Pc}_{\ge 1})$ (top right) from Lemma \ref{lem:functor2}.}
\label{fig:upper-truncation}
\end{figure}

In terms of hyperplane arrangement, any upper truncation of an MAT-free graphic arrangement is again MAT-free. 
This fact is not true in general. 
For example, the Weyl subarrangement defined by the $1$-upper truncation of the root poset of type $B_3$ is not free, hence not MAT-free.
It would be interesting to find for which MAT-free  arrangement or for which upper truncation of a given MAT-free  arrangement this property holds true.

  %********************************************************************************************************
 
\subsection{From MAT-labeled graphs to LR-vines}
\label{subsec:appl-MAT-to-vine}

  %********************************************************************************************************

\subsubsection{Strongly chordal graphs and m-vines}
\label{subsubsec:SC-mvine}
Given a strongly chordal graph $G$, Zhu-Kurowicka \cite[\S3.4]{ZK22} showed that there exists an m-vine, equivalently, an LR-vine $\Pc$ (by our Theorem \ref{thm:LR-characterize}) such that the principal ideals generated by the maximal elements of $\Pc$ are in one-to-one correspondence with the maximal cliques of $G$. 
Their method is based on the existence of a \emph{strong clique tree} of $G$. 
We give below a different construction of such an LR-vine thanks to the equivalence between the LR-vines and MAT-labeled graphs from Theorem \ref{thm:1-to-1}. 

 		\begin{theorem}
	\label{thm:max=max-alt} 
Given a strongly chordal graph $G$, there exists an LR-vine $\Pc$ such that the principal ideals generated by the maximal elements of $\Pc$ are in one-to-one correspondence with the maximal cliques of $G$.
 
	\end{theorem}
 		
	\begin{proof} 
By Theorem \ref{thm:MAT=SC}, there exists an MAT-labeling $\lambda$ of $G$. 
The construction of such a $\lambda$ can be found in \cite[Theorem 5.12]{TT23} based on the notion of \emph{clique intersection poset} of $G$ first appeared in \cite{NR15}.
Let $ \Pc= \Psi(G,\lambda)  $ (Lemma \ref{lem:functor1}). 
Theorem \ref{thm:max=max-alt} is proved once we prove that  the set $\max( \Pc)$ of maximal elements of $ \Pc$ coincides with the set $\mcK(G)$ of maximal cliques of $G$. 
By Lemma \ref{lem:max-prin}, any maximal clique $C$ in $G$ is principal hence an element in $ \Pc$. 
Moreover, $C \in\max( \Pc)$. 
Otherwise, there exists a clique $C'  \in\max( \Pc)$ such that $C \subsetneq C'$. 
This contradicts the maximality of $C$.
Hence $\mcK(G) \subseteq \max( \Pc)$. 
The reserve inclusion is proved similarly. 
Thus $ \Pc$ is a desired LR-vine.
	\end{proof}

  %********************************************************************************************************

\subsubsection{Marginalization and sampling order}
\label{subsubsec:SO}

The notions of \emph{marginalization} and  \emph{sampling order} of an R-vine were introduced in \cite[\S3]{CKW15}. 
We give below  an extension of these notions to an LR-vine. 

\begin{definition}[Marginalization] 
\label{def:margin}
Let $(\Pc, \le_\Pc, \rk_\Pc)$ be an LR-vine. 
Let $v \in \min(\Pc)$ be a minimal node.
The \tbf{marginalization} $(\Pc\,\|\, v, \le_\Pc)$ of $\Pc$ by $v$ is the induced subposet of $\Pc$ obtained by removing  $v$ and the nodes whose conditioned sets contain $v$, i.e.
$$\Pc\,\|\, v:=\Pc \setminus (\{v\} \cup \{ x \in \Pc \mid v \in C_x\} ).$$
\end{definition}

Let $\Qc$ be an  induced subposet of a finite graded poset $(\Pc, \le_\Pc, \rk_\Pc)$. 
We say that $\Qc$ is \tbf{graded by $\rk_\Pc$} if the restriction $\rk_\Pc |_\Qc$ is a rank function on $\Qc$.
The marginalization of $\Pc$ by a  minimal node is not necessarily graded by $\rk_\Pc$ in general. 
For example, let $\Pc$ be the C-vine with $4$ minimal elements $1,2,3,4$ in Figure~\ref{fig:C-vine}. 
Then the  marginalization  $\Qc:=(\Pc\,\|\, 2, \le_\Pc)$  is not  graded by $\rk_\Pc$ since $\rk_\Pc(\textcolor{blue}{34|12})=4$ while $\textcolor{blue}{34|12}$ cannot have rank $4$ in $\Qc$.

\begin{definition}[Sampling order] 
\label{def:S-order}
Let $(\Pc, \le_\Pc, \rk_\Pc)$ be an $\ell$-dimensional  (L)R-vine.
An ordering $ (v_{1}, \dots, v_{\ell}) $ of minimal nodes in $ \Pc$ is a \textbf{sampling order} if the marginalization $\Pc_{i}:= \Pc_{i+1} \,\|\,  v_{i+1}$ is an (L)R-vine graded by $\rk_\Pc$ for each $1 \le i \le \ell-1$. Here we let $\Pc=\Pc_\ell$.
\end{definition}

It is shown in  \cite[Theorem 5.1]{CKW15} that an R-vine always has a sampling order. 
Now we generalize this result to LR-vines. 

\begin{theorem} 
\label{thm:MATPEO-SO}
 Let $\Pc $ be an $\ell$-dimensional  LR-vine and let $ (G,\lambda) = \Omega( \Pc)$ (Lemma  \ref{lem:functor2}).
 If $ (v_{1}, \dots, v_{\ell}) $ is an MAT-PEO of $ (G,\lambda)$, then $ (v_{1}, \dots, v_{\ell}) $ is a sampling order  of   $\Pc$. 
 As a consequence, an LR-vine always has a sampling order.
\end{theorem}

   \begin{proof}
   The assertion is clearly true when $\ell\le1$. 
   Suppose that $\ell \ge 2$ and $ (v_{1}, \dots, v_{\ell}) $ is an MAT-PEO of $ (G,\lambda)$. 
  By Lemma \ref{lem:MAT-simplicial}, $ (G_{i}, \lambda_{i}) $ is an MAT-labeled graph for each $1\le i \le \ell -1 $, where $ G_{i}:= G[\{v_{1}, \dots, v_{i}\}] $ and $ \lambda_{i}:= \lambda|_{E_{G_{i}}} $. 
By the proof of Theorem \ref{thm:MAT-to-poset},  for each $i$,  $\Psi (G_{i}, \lambda_{i})=\widehat{\Pc_{i}}$ where $ \Pc_{i}:= \Pc_{i+1} \,\|\,  v_{i+1}$, $\Pc_\ell:= \Pc$ and $ \widehat{\Pc_{i}}$ is the poset isomorphic to  $ \Pc_{i}$ in Proposition \ref{prop:Phat}. 
  Thus each  $ {\Pc_{i}}$ is   an  LR-vine graded by $\rk_\Pc$. 
  Hence $ (v_{1}, \dots, v_{\ell}) $ is a sampling order  of   $\Pc$. 
  The consequence is straightforward since an MAT-labeled graph always has an  MAT-PEO by Theorem \ref{thm:MAT-PEO}.
\end{proof}

	\begin{remark} 
			\label{rem:MATPEO-SO} 
The converse of the main assertion in Theorem \ref{thm:MATPEO-SO} is not true in general. 
Namely, a sampling order of $\Pc$ is not necessarily an MAT-PEO of $ (G,\lambda)$. 
The reason is that even if the marginalization $\Pc\,\|\, v$ for some $v \in \min(\Pc)$  is  an LR-vine graded by $\rk_\Pc$, the node $v$ is not necessarily an MAT-simplicial vertex of $ (G,\lambda)$. 
For example, let $\Pc$ be the $3$-lower truncation of the C-vine in Figure~\ref{fig:C-vine} (see Figure \ref{fig:inclined} below). 
The associated MAT-labeled graph $ (G,\lambda)= \Omega( \Pc)$ is the graph obtained from the corresponding complete graph by removing the edge $\{v_3,v_4\}$.
Then the  marginalization  $\Pc\,\|\, 2 $   is an  LR-vine graded by $\rk_\Pc$, but the vertex $v_2$ is not (MAT-)simplicial in $ (G,\lambda)$ since its neighborhood does not form a clique. 
In addition, $(1,3,4,2)$ is a sampling order of $\Pc$, but $(v_1,v_3,v_4,v_2)$ is not  an MAT-PEO of $ (G,\lambda)$. 
	\end{remark}
	
While upper or lower truncation of an LR-vine can be visualized as ``horizontal" truncation, the marginalization is a ``vertical" truncation. 
 Figure \ref{fig:inclined} depicts the $3$-lower truncation $\Pc$ of the C-vine in Figure~\ref{fig:C-vine} and the marginalization  $\Pc\,\|\, 4 $. 
We may continue marginalizing and get the sampling order  $(1,2,3,4)$ of  $\Pc$. 
Furthermore, $(v_1,v_2,v_3,v_4)$ is an MAT-PEO of the MAT-labeled graph associated to $\Pc$.

\begin{figure}[htbp!]
   \centering
\begin{tikzpicture}[scale=1]
 
   \node (x1) at (1.5,-0.5) {};
    \node (y1) at (-2,3){};
 
  \node (24) at (0,2) {\textcolor{red}{$24|1$}};
  \node (23) at (-2,2) {\textcolor{red}{$23|1$}};
  \node (14) at (1,1) {$14$};
  \node (13) at (-1,1) {$13$};
   \node (12) at (-3,1) {$12$};
         \node (4) at (2,0) {$4$};
       \node (3) at (0,0) {$3$};
     \node (2) at (-2,0) {$2$};
   \node (1) at (-4,0) {$1$};
   
  \draw (1)--(12) -- (23) ;
    \draw (4)--(14) -- (24)  ;
    \draw  (3)--(13) -- (23) ;
        \draw (2)-- (12) -- (24) ;
          \draw (1)--(13)  ;
           \draw (1)--(14)  ;
                         \draw[ForestGreen, thick] (x1) -- (y1) ;

\end{tikzpicture}
\caption{\small An LR-vine and its marginalization visualized as ``vertical" truncation.}
\label{fig:inclined}
\end{figure}
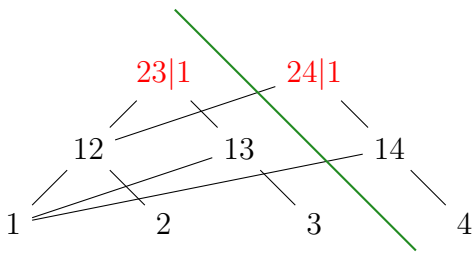

  %********************************************************************************************************

\subsubsection{Number of ideals of a D-vine}
\label{subsubsec:D-vine}
Given an arbitrary R-vine $\Pc$, a natural question is to find the number of m-vines or ideals in $\Pc$. 
This question is still open in general (see e.g., \cite[\S4]{ZK22}). 
However, in the particular case of D-vine, we have an explicit answer owing to a classical result of Shi \cite{Shi97} in the theory of hyperplane arrangements.

\begin{theorem} 
\label{thm:ideal-Catalan}
 The number of ideals in the D-vine of dimension $\ell-1$  for $\ell\ge2$ is given by the $\ell$-th Catalan number $\mathrm{Cat}_\ell ={\frac {1}{\ell +1}}{2\ell \choose \ell } =\prod \limits _{k=2}^{\ell}{\frac {\ell +k}{k}}.$
 
\end{theorem}

   \begin{proof}
   By Remark \ref{rem:D-vine-RP}, the D-vine of dimension $\ell-1$ is isomorphic to the type $A_{\ell-1}$ root poset. 
 It is known that the number of ideals in the latter is given by the  $\ell$-th Catalan number $\mathrm{Cat}_\ell$ \cite[Theorem 1.4]{Shi97}.
\end{proof}

Some calculation on C-vines suggests us the following conjecture:

\begin{conjecture} 
\label{thm:ideal-C}
 The number of ideals in the C-vine of dimension $\ell\ge1$ is given by the OEIS sequence A047970 $N_\ell = \sum_{i=0}^\ell( (i+1)^{\ell-i} - i^{\ell-i}).$
 
\end{conjecture}

 \vskip 1em
\noindent
\textbf{Acknowledgements.} 
The second author was supported by a postdoctoral fellowship of the Alexander von Humboldt Foundation at Ruhr-Universit\"at Bochum. 
We thank Takuro Abe, Michael DiPasquale, Emanuele Delucchi, and Max Wakefield for interesting questions and comments. 

 \bibliographystyle{abbrv}
\bibliography{references}

\end{document}